\documentclass{amsart}
\usepackage{amsfonts,amssymb,amsmath,amsthm}
\usepackage{url}
\usepackage{enumerate}
\usepackage[active]{srcltx}

\newtheorem{theorem}{Theorem}[section]
\newtheorem{lemma}[theorem]{Lemma}
\newtheorem{proposition}[theorem]{Proposition}

\newtheorem{example}[theorem]{Example}
\theoremstyle{definition}
\newtheorem{definition}[theorem]{Definition}
\newtheorem{remark}[theorem]{Remark}
\numberwithin{equation}{section}

\newcommand{\X}{{\mathfrak X}}

\renewcommand{\H}{{\mathcal H}}
\def\C{\mathbb C}
\def\R{\mathbb R}
\def\S{{\mathcal S} }

\def\C{\mathbb C}
\def\R{\mathbb R}
\def\I{\mathbb I}
\def\N{\mathbb N}
\def\al{\alpha}

\def\be{\beta}

\def\de{\delta}
\def\rh{\rho}
\def\et{\eta}

\def\GA{\Gamma}

\def\ve{\varepsilon}

\def\la{\lambda}
\def\OM{\Omega}

\def\va{\varphi}
\def\ta{\tau}

\def\va{\varphi}

\def\g{\mathfrak{g}}

\def\h{\mathfrak{h}}

\def\p{\mathfrak{p}}

\def\f{\mathfrak{f}}

\def\z{\mathfrak{z}}

\def\g{\mathfrak g}
\def\h{\mathfrak h}

\def\la{\lambda}
\def\ve{\varepsilon}
\def\si{\sigma}

\def\ga{\gamma}
\def\ph{\phi}
\def\ch{\chi}
\def\ta{\tau}
\def\ps{\psi}
\def\N{\mathbb{N}}
\def\Z{\mathbb{Z}}
\def\R{\mathbb{R}}
\def\C{\mathbb{C}}

\def\ol#1{\overline{#1}}
\def\nn{\nonumber}
\def\noop#1{\Vert #1\Vert_{\rm op}}

\def\R{{\mathbb R}}
\def\C{{\mathbb C}}
\def\N{{\mathbb N}}

\def\Z{{\mathbb Z}}
\def\T{{\mathbb T}}
\def\I{{\mathbb I}}
  \def\Id{{\mathbb I}}

\def\B{{\mathcal B}}

\def\F{{\mathcal F}}

\def\H{{\mathcal H}}
\def\P{{\mathcal P}}
\def\K{{\mathcal K}}

\def\X{{\mathcal X}}

\def\O{{\mathcal O}}

\def\Ad{{\text Ad}}
\def\tr{{\text tr}}

\def\iy{\infty}

\def\ol#1{\overline{#1}}

\def\hb#1{\hbox{#1}}
\def\val#1{\vert #1\vert}

\def\no#1{\Vert #1\Vert }

\def\wh#1{\widehat{#1}}

\def\exp#1{{\rm exp} #1}
\def\ind#1#2{\hb{ind}_{#1}^{#2}}

\def\ker#1{\hb{ker}(#1)}

\def\res#1{_{\vert #1}}
\def\inv{^{-1}}

\def\es{\emptyset}

\def\hb#1{\hbox{#1}}
\def\val#1{\vert #1\vert}

\def\ker#1{\hb{ker}(#1)}

\def\dim#1{\hb{dim}(#1)}

\def\L1#1{L^1(#1)}

\def\L#1#2{L^{#1}(#2)}
\def\l#1#2{L^{#1}(#2)}

\def\ti{\times }

\def\lef({\left(}
\def\rig){\right)}

\author{ Hedi REGEIBA\and Jean LUDWIG  }
\address{
Universit\'e de Lorraine\\
Institut Elie Cartan de Lorraine\\
UMR 7502, Metz, F-57045, France.}
\address{
CNRS\\
Institut Elie Cartan de Lorraine\\
UMR 7502, Metz, F-57045, France.}
\email{rejaibahedi@gmail.com.}
\email{jean.ludwig@univ-lorraine.fr.}
\begin{document}
\title[$C^\ast$-alg with no con du lim and nil Lie gro.] {  $C^\ast$-algebras with norm controlled dual limits and   nilpotent Lie groups.}

\date{}    

\maketitle  

\begin{abstract}
Motivated by the description of the $ C^* $-algebras of $5$ dimensional nilpotent Lie groups as algebras 
of operator fields defined over their spectra, we introduce the family 
of    $C^*$-algebras with norm controlled dual limits  and we  show that  the $C^*$-algebras of the 
$5$ dimensional nilpotents Lie  groups  belong to this class. 
\end{abstract}

\section{Introduction.}\label{intro}
\subsection{}

In recent papers,  the $ C^*$-algebra of the Heisenberg groups,
of  threadlike groups and $ ``ax+b'' $-like groups
 have been described as algebras of operator fields (see  
\cite{Lud-Tur} and  \cite{Lin-Lud}). For this description a precise understanding of the topology of the spectrum  
of these groups was essential (see for instance \cite{Arc-Lud-Sch} for the case of threadlike groups).
In this paper we study  the group $ C^* $-algebra of all
connected nilpotent Lie groups of dimension $ \leq 5 $ as algebra
of operator fields. This family  of Lie
groups has been classified by several authors, a list can be found for instance   in \cite{Nie}. It contains the Heisenberg groups
of dimensions 3 and 5 and also the threadlike groups $ F_4$ and  $F_5$ . There are $6$ simply connected   nilpotent un-decomposable 
Lie  groups of dimension $5$. 
Thanks to Kirillov's orbit picture of the spectrum of a connected simply
connected nilpotent Lie group,  we have an  description of the spectrum of these groups in terms of the structure of the space of its co-adjoint orbits. But the orbit theory is only an algorithm, it does not 
 give us  any details about the  result of computations. The topology of the orbit
space  or the behaviour of the operators $ \pi(F),\  F\in C^*(G)
$ as $ \pi  $ varies in the spectrum is different for each of these groups and must be studied case by case.\\ 
The  paper begins  with  section $ 2$, where some definitions, methods and results are presented which are needed in the sequel.
In section 3,  a  family of $C^*-$algebras, which we call $ C^* $-algebras with norm controlled dual limits (see Definition \ref{norcontspec}) is introduced. This is a family of separable $ CCR $-algebras $ A $, for which   there exists a 
finite increasing family $ S_0\subset S_1\subset\ldots\subset S_d=\wh A  $ of closed 
subsets of the spectrum $ \wh A $ of $ A $, such that for $ i=1,\cdots, d, $ the subsets $\GA_0=S_0$ and  $ \GA_i:=S_i\setminus S_{i-1} $ 
have separated   relative topologies and which have the property that for every converging sequence 
$ \ol\ga=((\ga_k,\H_k))_k \subset S_i$ with limit set $ L(\ol\ga)\subset S_{i-1} $ there exists a sequence $ (\tilde \si_{\ol\ga,k})_k:CB(S_{i-1})\mapsto B(\H_{k}) $ (here $ CB(S_{i-1}) $ denotes the $ C^* $-algebra of continuous bounded operator fields defined over $S_{i-1} $) of linear mappings, which is uniformly bounded in $ k $, such that for every   $ a\in A $  we have that $ \lim_k \noop{\ga_k(a)-\tilde\si_k(a)}=0 $. These $ C^* $-algebras  are then completely determined by the topology of their spectra (in particular by the limit sets $ L(\ol\ga) $ of properly converging sequences in $ \wh A $)  and these mappings $ (\tilde\si_{\ol\ga,k})_k $ 
(see Theorem \ref{aisdsta}).

We then study the 6 groups of dimension $\leq  $ 5 case by case and we show that all of them have $C^* $-algebras with norm controlled dual limits.
For the Heisenberg  and the threadlike groups this  has already be shown in the paper \cite{Lud-Tur}.
There remains then only the 4 groups $G_{5,2}, G_{5,3},G_{5,4} \text{ and }G_{5,6}$, which are treated separately in 
the sections 6,7,8 and 9. Since the structure of the dual space of these groups are different for each of them, we must 
determine the topology of $\wh G $ group by group and  construct  by hand for every limit set $\ol\O $ of a properly converging 
sequence in $\g^*/G $ these essential mappings $\si_{\ol\O,k} $.

To understand these mappings $\si_{\ol\O,k} $, one has to recall a theorem of Fell, (see  \cite{Fell-1}), 
where he shows that in the case of a properly converging net $\ol\O=(\pi_k)_k\subset \widehat A $ of a $C^* $ -algebra $A $, with limit set $L $, one has that 

\begin{eqnarray*}
 \lim_k\noop{\pi_k(a)}=\sup_{\pi\in L}\noop{\pi(a)}, a\in A.
 \end{eqnarray*}
To implement that theorem we need the mappings $\si_{\ol\O,k} $. We shall construct for our limit sets $L $ for every $k\in\N,  $ an increasing sequence of  countable subsets $L_k $ of $L $ and a sequence of positive numbers $\ve_k $ such that $\lim_k \ve_k =0 $ and such that $\bigcup_k L_k $ is dense in $L $. Furthermore for every $\pi\in L_k $ and $k\in\N $, we shall find an orthogonal projection  $P_{k,\pi} $ on the Hilbert space $ \H_k$ of $\pi_k $, such that $\sum_{\pi\in L_k} P_{k,\pi}=\Id_{\H_k} $ and a linear  mapping $U_{k,\pi}:\H_{\pi}\to \H_{\pi_k} $ such that 
\begin{eqnarray*}
 \sum_{\pi\in L_k}\noop{P_{k,\pi}\circ \pi_k(a)\circ P_{k,\pi}-U_{k,\pi}\circ\pi(a)\circ U_{k,\pi}^*}\leq \ve_k \no a, a\in C^*(G). 
 \end{eqnarray*}
 This results for the 5 dimensional groups  lead to the following question.
 
 Do  the $C^* $-algebras of connected nilpotent Lie groups have all this property of norm controlled dual limits?
 
 In several  forthcoming papers it will be shown that the answer to this question is yes for all groups of dimension 6.
\section{Preliminaries.}

\subsection{Orbit picture.}\label{kirrth}
\rm   Kirillov's orbit theory for a connected simply connected nilpotent Lie  group $ G $ tells
us that for every irreducible unitary representation $\pi$ of $G$ there   exists
an $\ell\in\g^*$ and a  polarization $\p\subset\g$ at $\ell$ (i.e. a subalgebra $ \p $ of $ \g $ of dimension 
$ d=\frac{\dim\g+\dim{\g(\ell)}}{2} $ with the property $ \langle{\ell },{[\p,\p] }\rangle=\{0\} $) such that
$\pi$ is equivalent to the induced representation
$\pi_{\ell,\p}=ind_P^{G}\chi_{\ell}$ of the unitary character
$\chi_{\ell}=e^{-2\pi i\ell\circ\log|_P}$  from  $P=\exp(\p)$ to $ G $. Furthermore for two
linear functionals $\ell,\ell'$ on $\g$  and the Pukanszky polarizations $\p$ at
$\ell$ $($ resp $\p'$ at $\ell')$, the representations $\pi_{\ell,\p}$ and
$\pi_{\ell',\p'}$ are equivalent if and only if $\ell$ and $\ell'$   are
contained in the same $G$-orbit (see \cite{Cor-Gre}). Let $[\pi]$ denote the unitary equivalence class of a unitary 
representation $\pi$ of $G$. The Kirillov  map $$K:\g^*/G\to\widehat G;
\Ad^*(G)\ell\to[\pi_{\ell,\p}]$$ is a  homeomorphism of the orbit space $ \g^*/G $ onto
$\widehat{G}$ (see \cite{Lep-Lud}).
\subsection{Some definitions and results.}\label{tools}
We indicate here some definitions, methods and results, which will be needed in the sequel.

\begin{enumerate}\label{}
\item  
\begin{definition}
\rm 
Let $ H=\exp(\h) $ be a closed connected subgroup of a connected nilpotent Lie  group $ G=\exp (\g) $ and let $
\ch_\ell:H\to \T;\ \ch_\ell(h)=e^{-2\pi i \langle\ell,\log(h)\rangle}, h\in H(\ell\in \g^*)$, be a unitary character of $ H $. The quotient space $ G/H
$ 
has a unique left invariant measure $ d\dot g $. With this measure we can
define 
the Hilbert space $ \H=\H_{H,\ell}=\l2{G/H,\ell} $  by
\begin{eqnarray}\label{}
\nn \l2{G/H,\ell}&:=&\{\xi:G\to\C, \xi \textrm{ mesurable, }
\xi(gh)=\ch_\ell(h\inv)\xi(g), h\in H,\ g \in G,\\
\nn &&\no \xi_2^{2}:=\int_{G/H}\val{\xi(g)}^{2}d\dot g<\iy\}. 
\end{eqnarray}
The group $ G $ acts by left translation on this space and defines  a unitary 
representation 
\begin{eqnarray}\label{}
\nn \si_{\ell,\h}(g)\xi(u)&:=&\xi(g\inv u), \xi\in\l2{G/H,\ell},g, u\in G 
\end{eqnarray}
called the induced representation of $ \ch_\ell $ (from $ H $ to $ G $).

If $\h$ is  polarization at $\ell$, then the representation $\si_{\ell,\h}$ is irreducible and we denote it sometimes by $\pi_{\ell,\h}$.
If we take a Malcev basis $\B=\{X_1,\cdots, X_d\}$ of $\g$ relative to $\h$, which means that the 
subspaces $\g_j:=\text{span}\{X_j,\cdots, X_d,\h\}$ is  a subalgebra of $\g$ and that $\g=\oplus _{j=1}^d \R X_j\oplus \h$ is a direct sum, 
then the mapping $E_\B:\R^d\ti H\mapsto G, E_\B(t_1,\cdots, t_d):=\exp(t_1 X_1)\cdots \exp (t_d X_d)h$ is  a diffeomorphism and it allows us 
to identify the Hilbert space  $L^2(G/H,\ell)$ with the space $L^2(\R^d)$. We shall consider in the following pages always the 
representations $\si_{\ell,\h}$ as representations on the space $L^2(\R^d)$.
\end{definition}
\rm  It is well known that  for $ F\in L^{1}(G) $ the operator $ \pi_{\ell,\h} $ 
is a kernel operator with kernel function
\begin{eqnarray}\label{}
\nn F_{\ell,\h}(s,t)&=&\int_H F(sht\inv)\ch_\ell(h)dh, s,t\in G. 
\end{eqnarray}
If $ H $ is  a normal subgroup of $ G $, then the function $ F_{\ell,\h} $ can be written as
\begin{eqnarray*}
F_{\ell,\h}(s,t)=\wh F^{\h}(st\inv, \Ad^*(t) \ell\res\h), s,t\in G,
\end{eqnarray*}
where
\begin{eqnarray*}
\wh F^\h(s,q):=\int_H  F(sh)e^{-2\pi i \langle{q},{\log (h)}\rangle}dh, s\in G, q\in \h^*.
\end{eqnarray*}

 We denote for a normal subgroup $H=\exp(\h)$ by $L_c^1=L^1_{c,\h}$ the subspace of $L^1(G)$ consisting of all $F$'s in
$L^1(G)$ for which $\wh F^\h\in C_c^{\iy}(G/H\times \h^*)$. This space $L_c^1$ is dense in $L^1(G)$ and
hence it is also dense in $C^*(G)$.
The functions $\wh F^\h$ satisfy the  covariance condition
\begin{eqnarray}\label{covacon}
 \nn\wh F^\h(sh,q) &= & \ch_q(h\inv)\wh F^\h(s,q), s\in G, h\in H, q\in \h^*.
\end{eqnarray}

The vector space $L^1_c$ is of course dense in $L^1(G)$ and hence also in $C^*(G)$.

 Since for every $F\in L^1_c$ the  function $\wh F^\h $ is smooth with compact support on $G/H\ti \h^*$,  there exists  a function $\va\in C_c(G/H)$ such that 
\begin{eqnarray*}
 \val{\wh F^\h(s,q)}\leq \no{q}\val{\va(s)}, s\in G, q\in\h^*.
 \end{eqnarray*} 
\item  We shall often use in the sequel  the following fact.

\rm Let $ F:\R^d\ti\R^d\to \C $ be a smooth function with compact support for which there exists some continuous function $ \va:\R^d\to \R_+ $ with compact support, such that 
\begin{eqnarray*}
  \val {F(x,y)}\leq \va(x-y), x,y\in\R^d .
 \end{eqnarray*}
Then by Young's inequality, for  the kernel operator $ T_F $ defined on $ L^2(\R^d) $ by
\begin{eqnarray*}
T_F(\xi)(x):=\int_{\R^d}F(x,y)\xi(y)dy \text{ for } \xi\in L^2(\R^d), x\in\R^d,
\end{eqnarray*}
its operator norm $ \noop{T_F} $ is bounded by the $ L^1 $-norm $ \no \va_1 $ of $ \va $.

\item  Let $G$ be  a second countable locally compact group with  a continuous action by homeomorphisms $G\ti S\to S, (g,s)\to g\cdot s$  on a second countable locally compact space $S$. Denote for $s\in S$ the stabilizer  of $s$  in $G$ by $G_s$.

Let $\ol\O=(\O_k)$ be  a sequence of $G$-orbits in $S$.  We say that $\ol\O$ \textit{converges with multiplicity 1} to the $G$-orbit $\O$, if there exists for every $k\in\N$ an element $s_k\in \O_k$ and an $s\in \O$, such that $\lim_k s_k=s$ and such that for any compact subset $K$ of $ S$, for which  the intersection of the interior $\dot K$ of $K$ with $G\cdot s$ is not empty,   there exists a compact subset $C\subset G$, such that for $k$ large enough, for any  $g\not\in C G_{s_k}$ we have that $g\cdot s_k\not \in K $.
 
\item Let $\g $ be a nilpotent Lie algebra. We fix an euclidean  scalar product on $\g $.  
Let $(\ell_k)_{k\in\N} $ be a converging sequence in $\g^* $ with limit $\ell $. 
Let $(\h_k)_k $ be a sequence of subalgebras of $\g $, such that $\dim{\h_k}=n-d, k\in\N,   $ for some $d\in\N^* $, and such that $\langle{\ell_k},{[\h_k, \h_k]}\rangle=\{0\}, k\in\N $. 
We pick for every $k\in\N $ an orthonormal  Malcev basis of $\B_k=\{Z_1^{k},\cdots, Z_n^k $ $ \} $ such that $\h_k=\text{span}\{Z^k_{n-d},\cdots, Z_n^k\} $ and we can assume (passing if necessary to a subsequence), that the vectors $Z^k_j $ converge for every $j=1,\cdots, n $ to a vector $Z_j $. 
Then $\B:=\{Z_1,\cdots, Z_n\} $ is an orthonormal Malcev basis of $\g $ which passes through $\h:=\text{span}\{Z_{n-d},\cdots, Z_n\} $. Then $\h $ is  subalgebra of $\g $ and  $\langle{\ell},{[\h,\h]}\rangle=\{0\} $.
Let $H_k:=\exp({\h_k}), k\in  \N $ and $H=\exp(\h) $. Let $\si_k:=\si_{\ell_k,\h_k} $  and 
$\si=\si_{\ell,\h} $. This gives us the following 
\begin{proposition}\label{weakcon}
The representations $\si_k $ converge weakly to the representation $\si $. This means that for every $\xi,\et \in L^2(\R^d) $, for every $a\in C^*(G) $:
\begin{eqnarray*}
 \lim_{k\to\iy}\langle{\si_k(a)\xi},{\et}\rangle=\langle{\si(a)\xi},{\et}\rangle.
 \end{eqnarray*}
 \end{proposition}
\begin{proof}Take first $F\in C_c(G) $ and $\xi,\et\in C_c(\R^d) $. 
We have then that:
\begin{eqnarray*}
 \langle{\si_k(F)\xi},{\et}\rangle&=&\int_{\R^d}\int_{\R^d} \xi(u)F_k(u,v)\ol{\et}(v)du dv,
\end{eqnarray*}
where
\begin{eqnarray*}
 F_k(u,v):=\int_{\R^{{n-d}}}F(E_{k}(u)E_{k,d}(h)E_k(v)\inv)\ch_k(E_{k,d}(h))dh, u,v\in\R^n
 \end{eqnarray*}
and where $E_k=E_{\B_k},E_{k,d}=E_k{\res {\R^{n-d}}},k\in\N $. 
We get  similar expression for $k=\iy $, i.e. for $E_\B $ etc. 
It is easy to see that the  supports of the functions $f_k:\R^{d}\ti\R^d\ti \R^{{n-d}}  $ defined by
\begin{eqnarray*}
 f_k(u,v,h):=\xi(u)F(E_{k}(u)E_{k,d}(h)E_k(v)\inv)\ol{\et(v)},
 \end{eqnarray*}
are contained in a common compact set and that the functions  converge point-wise to the  function
\begin{eqnarray*}
 f(u,v,h):=F(E(u)E_{d}(h)E(v)\inv)\ch(E_{d}(h)),
 \end{eqnarray*}
where $E_d=E_{\B}{\res{\R^{n-d}}}.$ Therefore by Lebegue's theorem of dominated convergence we have that
\begin{eqnarray*}
 \lim_{k\to\iy}\langle{\si_k(F)\xi},{\et}\rangle&=&\langle{\si(F)\xi},{\et}\rangle.
\end{eqnarray*}
The proposition now follows from the density of $C_c(G) $ in $C^*(G) $ and the density of $C_c(\R^d) $ in $L^2(\R^d) $. 
\end{proof}
\begin{theorem}\label{conmult1}
Let $(\pi_k)_k\subset\wh G $ be a sequence  which converges with multiplicity 1 to a single limit point $\pi $. Then there exists  
a subsequence (also indexed by $k $ for simplicity of notations), a Hilbert space $\H$,  for every $k\in\N $ a concrete realization $(\si_k,\H) $ of $\pi_k $ and a concrete realization $(\si,\H) $ of $\pi $, such that 
\begin{eqnarray*}
 \lim_k \noop{\si_k(a)-\si(a)}=0, a\in C^*(G).
 \end{eqnarray*}

 \end{theorem}
\begin{proof} We write $\pi_k\simeq \si_k, k\in\N$, where $\si_k=\ind {P_k}G \ch_{\ell_k},k\in\N, $ where $\ell_k $ is an element in the Kirillov orbit of $\pi_k $ and where $P_k=\exp{(\p_k)} $ is a polarization at $\ell_k $. Since $\pi_k $ converges to $\pi $, We can assume that $\lim_k\ell_k =\ell$ for some $\ell $ in the orbit of $\pi $ and that (passing to a subsequence) $\lim_k\p_k =\p $ in the sense of the preceding Proposition \ref{weakcon}. Then $\p $ is  a polarization at $\ell $, since $\pi $ is the only limit of the sequence $(\pi_k)_k $.

Let $j(\es) $ be the minimal dense ideal of $C^*(G) $ i.e. the Pedersen ideal of $C^*(G)$.
Since $(\pi_k)_k $ converges with multiplicity 1 to $\pi $ it follows from  
\cite{L} that for very $a\in j(\es) $ the operators $\si_k(a), k\in\N, $ and 
$\si(a) $ are finite rank and that 
\begin{eqnarray*}
 \lim_k \tr(\pi_k(a))=\tr(\pi(a)).
 \end{eqnarray*}
Take now $a=a^*\in j(\es) $.
Then 
\begin{eqnarray*}
& &\Vert \si_k(a)-\si(a)\Vert_{H-S}^2\\&=&\tr(\si_k(a)^2)+\tr(\si(a^2))-\tr(\si_k(a)\circ \si(a))-\tr(\si(a)\circ \si_k(a) )\\
 &=&\tr(\si_k(a)^2)+\si(a^2)-2\tr(\si_k(a)\circ \si(a)))\\
 &\to &\tr(\si(a)^2)+\si(a^2)-2\tr(\si(a)\circ \si(a)))
 \\
 &=&0,
 \end{eqnarray*}
by Proposition \ref{weakcon}, the continuity of the trace  and the fact that $\si(a) $ has finite rank. Hence
\begin{eqnarray*}
 \lim_k \noop{\si_k(a)-\si(a)}=0.
 \end{eqnarray*}
The theorem now follows from the density of $j(\es) $ in $C^*(G) $.

\end{proof}

\item  
Denote for a measurable subset $ S\subset X $ of  a measure space $(X,\mu)$ 
 the multiplication operator with the indicator of a   measurable subset  $ S $ on $ L^2(X,\mu) $ by $ M_{S} $.
\begin{proposition}\label{sumsofop}
Let $(X,\mu)$ a measure space, let $(\si_i)_{i\in I}$ be a family of bounded linear operators on the Hilbert 
space $\H=L^2(X,\mu)$, such that $\noop{\si_i}\leq C,\ \forall i\in I,$ for some $C>0$. Suppose furthermore that there exists 
families $(T_{i,j})_{i\in I}(j=1,\cdots, N)$ and $(S_i)_{i\in I}$ 
of measurable subsets of $X$ such that $T_{i,j}\cap  T_{i',j}=\es,(j=1,\cdots, N),\  S_{i}\cap S_{i'}=\es$ whenever $i\ne i'$. Then the linear operator
\begin{eqnarray*}
 \si=\sum_{j=1}^N\sum_{i\in I}M_{T_{i,j}}\circ \si_i\circ M_{S_i}
 \end{eqnarray*}
is bounded by $NC$.

 \end{proposition}
\begin{proof} Let us write 
\begin{eqnarray*}
 \si^j:=\sum_{i\in I}M_{T_{i,j}}\circ \si_i\circ M_{S_i},\  j=1,\cdots, N.
 \end{eqnarray*}
Then $\si=\sum_{j=1}^N\si^j$ and 
for $\xi\in L^2(X,\mu),j\in \{1,\cdots, N\}$ we then have:
\begin{eqnarray*}
 \no{\si^j(\xi)}_2^2&=&\int_X\left|\sum_{i\in I}M_{T_{i,j}}\circ \si_i\circ M_{S_i}(\xi)(x)\right|^2d\mu (x)\\
 &\leq&\sum_{i\in I}\int_{T_{i,j}}\val{\si_i(M_{S_i}(\xi))(x)}^2d\mu (x)\\
  &\leq &\sum_{i\in I}\int_X\val{\si_i(M_{S_i}(\xi))(x)}^2d\mu (x)\\
 &\leq&\sum_{i\in I}C^2\int_X\val{ M_{S_i}(\xi(x))}^2d\mu (x)\\
 &=&C^2\left(\sum_{i\in I}\int_{S_i}\val{\xi(x)}^2d\mu (x)\right)\\
 &\leq&C^2\int_X\val{\xi(x)}^2d\mu (x)\\
 &=&C^2\no\xi_2^2.
 \end{eqnarray*}
Hence $\noop\si\leq N C$.
\end{proof}
\item  
\begin{remark}\label{intmapp}
\rm   Let $(X,\mu), (Y,\nu)$ be two measure spaces. Let $Y\to L^2(X,\mu); y\to \xi(y))$ be  an integrable mapping. 
Then we have that:
\begin{eqnarray}\label{estfh2}
 \nn\left\Vert\int_Y \xi(y)d\nu(y)\right\Vert_2&\leq& \int_Y \no{\xi(y)}_2d\nu(y)
 \\ 
 \nn&\text{i.e.}\\
 \nn\left(\int_X\left\vert{\int_Y\xi(y)(x)d\nu(y)}\right\vert^2d\mu(x)\right)^{1/2}&\leq&\int_Y\left(\int_X\vert \xi(y)(x)\vert^2d\mu(x)\right)^{1/2}d\nu(y).
 \end{eqnarray}

 \end{remark}

\item  \begin{definition}\label{conin}
\rm   We say that a net$(\ga_i)_{i\in\I}$ in  a topological space $\GA$ goes to infinity, if the net contains no converging subnet.

In particular, a sequence of orbits $\ol \OM=(\OM_k)_{k\in\N}\subset \g^*$ goes to infinity, if for any compact subset $K\subset \g^*$ 
there exists an index $k_0$ such that $K\cap \OM_k=\es$ whenever $k\geq k_0$.
 \end{definition}
\begin{proposition}\label{infycon}(Riemann-Lebesgue Lemma)
Let $A$ be a $C^*$-algebra. If a net $(\pi_k)_k\subset \wh A$ goes to infinity, then $\lim_k \noop{\pi_k(a)}=0$ for all $a\in A$.  
 \end{proposition}
\begin{proof} We know from \cite{Dix} Proposition 3.3.7,  that for every $ c>0$ and $a\in A$, the subset 
$\left\{\pi\in\wh A; \noop{\pi(a)}\geq c\right\}$ is quasi-compact. This shows that $\lim_k \noop{\pi_k(a)}=0$, if the net $(\pi_k)_k$ goes to infinity.
\end{proof}

 \item  Let $G$ be a locally compact group and $H$ a closed normal subgroup of $G$. Then the canonical projection $P_{G/H}:L^1(G)\mapsto L^1(G/H)$ defined by:
 \begin{eqnarray*} 
 P_{G/H}F(x):=\int_H F(xh)dh, F\in L^1(G), x\in G,
\end{eqnarray*}
is a surjective homomorphism (see \cite{Re}). 
Let 
\begin{eqnarray*}
 H^\perp=\left\{\pi\in \wh G, \pi(H)=\Id_{\H_\pi}\right\}.
 \end{eqnarray*}
Then
\begin{eqnarray*}
 \ker {P_{G/H}}=\left\{F\in L^1(G), \pi(F)=0,\pi\in H^\perp\right\}.
 \end{eqnarray*}
 The mapping $P_{G/H}$ extends then to a surjective $*$-homomorphism (also denoted by $P_{G/H}$) of $C^*(G)$ onto $C^*(G/H).$
 \end{enumerate}
\section{A special class of $ C^* $-algebras.}\label{spc}

\begin{definition}\label{fourtrans}
\rm Let $A$ be a $C^*-$algebra with spectrum $\wh A.$ We choose for every $\ga\in\wh A$ a representation 
$(\pi_\ga,\H_\ga)$ in the equivalence class $\ga$.
Let $l^{\iy}(\wh A)$ be the algebra of all bounded operator fields defined over $\wh A$ by
$$ l^\iy(\widehat A):=\left\{\ph=(\ph(\pi_\ga)\in\B(\H_{\ga}))_{\ga\in\widehat A},
\no \ph_\iy
:=\sup_{\ga}\noop {\ph(\pi_\ga)}<\iy\right\}.$$
We define for $a\in A$ its Fourier transform $ \F(a)=\wh a $ by:
\begin{eqnarray*}
\F(a)(\ga)=\wh a(\ga):=\pi_\ga(a),\ \ga\in\wh A.
\end{eqnarray*}
Then $\wh a$ is a bounded field of operators over $\wh A$, i.e. $\wh a\in l^\iy(\wh A)$. The mapping
\begin{eqnarray*}
 \F:A\to l^\iy(\wh A);\ a\mapsto \wh a
\end{eqnarray*}
is a isometric $*$-homomorphism. 
 \end{definition}
For a closed subset $ S\subset\wh A $, denote by $ CB(S) $ the $ * $-algebra of all uniformly 
bounded operator fields  $ (\ps(\ga)\in \B(\H_i))_{\ga\in S\cap \GA_i, i=1,\cdots, d}$, which are 
operator norm continuous on the subsets $ \GA_i \cap S$ for every $  i\in\{1,\cdots, d\} $ for which  $ \GA_i\cap S\ne\es $. 
We provide the algebra $ CB(S) $ with the infinity-norm:
\begin{eqnarray*}
\no{\va}_{S}:=\sup_{\ga\in S}\noop{\va(\ga)}.
\end{eqnarray*}
\begin{definition}\label{}
\rm  
Let $ A $ be a separable CCR $ C^* $-algebra, i.e. for every irreducible representation 
$ (\pi,\H)  $ of $ A $, the image of $ \pi $ is  the algebra of compact operators $ \K(\H) $.
 We suppose that there exists a 
finite increasing family $ S_0\subset S_1\subset\ldots\subset S_d=\wh A  $ of closed 
subsets of the spectrum $ \wh A $ of $ A $, such that for $ i=1,\cdots, d, $ the subsets $\GA_0=S_0$ and  
$ \GA_i:=S_i\setminus S_{i-1},\ i=1,\ldots,d,$  
are  Hausdorff in  their relative topologies. Furthermore we assume that for every $ i\in \{0,\cdots, d\} $ 
there  is a Hilbert space $ \H_i $ and for every $ \ga\in \GA_i $  a concrete realization $ (\pi_\ga,\H_i) $ of $ \ga $ on the Hilbert space $ \H_i $. 
The set $ S_0 $ is the collection $ \X $ of all 
characters of $ A $. 
 \end{definition}
\begin{definition}\label{norcontspec}
\rm   
We say that our $ C^* $-algebra $ A $ has   \textit{norm controlled dual limits } (\textit{NCDL} for short),  if for every $ a\in A $:
\begin{enumerate}\label{}
\item  The mappings $ \ga\to \F (a)(\ga) $ are norm continuous  on the different sets $ \GA_i $.

\item   For any  $ i=0,\cdots, d, $  and for any 
converging sequence contained in $ \GA_i $ with limit set  outside $ \GA_i $, there exists a properly converging sub-sequence $\ol\ga=(\ga_k)_{k\in\N} $,  a $ C>0 $ and for every $ k\in\N $  a linear mapping $ \tilde\si_{\ol\ga,k}: BC(S_{i-1})\to \B(\H_i)$, {which is  bounded by $ C\no{}_{S_{i-1}} $}, such that
\begin{eqnarray*}
&\lim_{k\to\iy }\noop{
\F (a)(\ga_k)-\tilde\si_{\ol\ga,k}(\F (a)\res{S_{i-1}}
}=0, \\ 
&\lim_{k\to\iy }
\noop{\tilde\si_{\ol\ga,k}(\F (a)\res{S_{i-1}})^*-\tilde\si_{\ol\ga,k}(\F (a)^*\res{S_{i-1}})}=0.
\end{eqnarray*}
(By the condition 1) the restriction of $ \F(a) $ to the limit set $ L(\ol\ga) $ is contained in $ CB(S_{i-1}) $. 
 \end{enumerate}
 \end{definition}
\begin{definition}\label{dnormconspec}
\rm   
 Let $ D^*(A) $ be the set of all operator fields $ \va $ defined over $ \wh A $ such that
\begin{enumerate}\label{}
\item $ \va(\ga)\in \K(\H_i) $ for every $ \ga\in\GA_i, i=1,\cdots, d $.
\item The field $ \ga $ is uniformly bounded, i.e. we have that $ \no\ga:=\sup_{\ga\in\wh A}\noop{\va(\ga)}<\iy $.
\item  The mappings $ \ga\to \va(\ga) $ are norm continuous on the different sets $ \GA_i $.
\item  We have for any sequence $ (\ga_k)_{k\in\N} \subset \wh A$ going to infinity, 
that  $ \lim_{k\to\iy}\noop{\va(\ga_k)}=0 $.
\item   Let $\ol\ga=(\ga_k)_{k\in\N} $ and $ \tilde\si_{k,\ol\ga},k\in\N, $ as in the point 2. of Definition \ref{norcontspec}. 
Then the restriction of $ \va $ to the subset $ L(\ol\ga) $ is contained in $ CB(S_{i-1}) $ by conditions 2. and 3.
and  we assume  that:
\begin{eqnarray*}
&\lim_{k\to\iy }
\noop{
\va(\ga_k)-\tilde\si_{\ol\ga,k}(\va\res{S_{i-1}})}=0, \\ 
&\lim_{k\to\iy }\noop{\tilde \si_{\ol\ga,k}(\va\res{S_{i-1}})^*-\tilde\si_{\ol\ga,k}(\va^*
\res{S_{i-1}})}=0.
\end{eqnarray*} 
 \end{enumerate}

 \end{definition}
We remark immediately   that for every $ a\in A $, the operator field 
$ \F(a) $ 
is contained in the set $ D^*(A) $. 
It turns out that $ D^*(A) $ is  a $ C^* $-sub-algebra 
of $ l^\iy(\wh A) $ and that 
$ A $   is isomorphic to $ D^*(A).$
\begin{theorem}\label{aisdsta}
Let $ A $ be a separable $ C^* $-algebra with  norm controlled dual limits. Then the subset 
$ D^*(A) $ of the $ C^* $-algebra $l^\iy(\wh A)  $ is a $ C^* $-sub-algebra of $l^\iy(\wh A)$ which is isomorphic with $ A $ 
under the Fourier transform.
 \end{theorem}
\begin{proof}

We see that the conditions 1. to 4. imply that $ D^*(A) $ is a closed involution-invariant subspace 
of $ l^\iy(\wh A) $.
For $ i=0,\cdots, d,  $ let  $ D^*_i $ be the set of all operator fields defined over $ S_i $, satisfying conditions 1. to 5. on the sets  $ S_j,\ j=1,\cdots, i $. 
Then the sets $ D^*_i $ are closed sub-spaces of the $ C^* $-algebra $ l^\iy(S_i) $.
Let $ I_C $ be the closed two-sided ideal in $ A $ generated by the elements of the form $ ab-ba, a,b\in A $. 
Then the space of characters $ S_0=\X $ of $ A $ is the spectrum of $ A/I_C $ and  $ D^*_0$ equals the algebra  $ C_0(S_0) $ 
of continuous functions on $ S_0 $ vanishing at infinity by the conditions 1., 2., 3. and 4. 
Since $ \F(C^*(A))\res{S_0}=C_0(S_0) $  it follows that $D_0^*=\F(A)\res{S_0}$.\\
Let us assume now that for some $ 1\leq i <d$ we have that
$ D^*_j=\F(A)\res{S_j}, j=1,\cdots, i-1 $.  We shall prove then that $ D^*_i=\F(A)\res{S_i} $. We know already that 
$ \F(A)\res{S_i} $ is an algebra sitting in the closed sub-space  $ D^*_i \subset l^\iy(\S_i)$ and it follows from its definition that  the restriction of $ D_i^* $ to $ S_{i-1} $ is contained in $ D^*_{i-1} $. 
Let $ \va,\ps\in D^*_i $.  By our assumption, there exists $ a,b\in A $ such that $ \va\res{S_{i-1}}=\wh a\res{S_{i-1}} $ 
and $ \ps\res{S_{i-1}}=\wh b\res{S_{i-1}} $. 
The product $ \va\circ \ps$ satisfies then also the conditions 1.to 4. 
for $ i $. We shall show that it too satisfies  condition $ 5. $ for $ i. $
Indeed  we have for any properly converging sequence $ (\ga_k)_k\subset \GA_i $ with 
limit set outside $ \GA_i $ that:
\begin{eqnarray*}
&&\noop{\va(\ga_k)\circ \ps(\ga_k)-\tilde\si_{\ol\ga,k}(\va\circ\ps\res {S_{i-1}})}\\
&=&\noop{\va(\ga_k)\circ \ps(\ga_k)-\tilde\si_{\ol\ga,k}(\va\res {S_{i-1}}\circ\ps\res {S_{i-1}})}\\
\nn &=&\noop{\va(\ga_k)\circ \ps(\ga_k)-\tilde\si_{\ol\ga,k}(\wh a\res {S_{i-1}}\circ\wh b\res {S_{i-1}})}\\
&\leq&\noop{\va(\ga_k)\circ \ps(\ga_k)-\tilde\si_{\ol\ga,k}(\va\res {S_{i-1}})\circ\tilde\si_{\ol\ga,k}(\ps\res {S_{i-1}})}\\
&+&\noop{\tilde\si_{\ol\ga,k}(\va\res {S_{i-1}})\circ\tilde\si_{\ol\ga,k}(\ps\res {S_{i-1}})-\wh a (\ga_k)\circ\wh b(\ga_k)}\\
&+&\noop{\wh a (\ga_k)\circ\wh b(\ga_k)-\tilde\si_{\ol\ga,k}(\wh a\res {S_{i-1}}\circ\wh b\res {S_{i-1}})}\\
&=&\noop{\va(\ga_k)\circ \ps(\ga_k)-\tilde\si_{\ol\ga,k}(\va\res {S_{i-1}})\circ\tilde\si_{\ol\ga,k}(\ps\res {S_{i-1}})}\\
&+&\noop{\tilde\si_{\ol\ga,k}(\wh a\res {S_{i-1}})\circ\tilde\si_{\ol\ga,k}(\wh{b}\res {S_{i-1}})-\wh a (\ga_k)\circ\wh 
b(\ga_k)}\\
&+&\noop{\wh {(ab)} (\ga_k)-\tilde\si_{\ol\ga,k}(\wh{ab}\res {S_{i-1}})}.
\end{eqnarray*}
This shows that $ \lim_{k\to\iy }\noop{\va\circ \ps(\ga_k)-\tilde\si_{\ol\ga,k}(\va\circ\ps\res {S_{i-1}})}=0 $ 
and so $ \va\circ \ps\in D^*_i $. Hence the subspace $ D^*_i $ is a $ C^* $-sub-algebra 
of $ L^{\iy}(S_i) $ containing the algebra $ \F(A)\res{S_i} $.
In order to prove that $ D^*_i=\F(A)\res{S_i} $, by  Stone-Weierstrass it suffices to show that the spectrum 
of $ \wh{D^*_i} $ equals that of $ \F(A)\res{S_i} $, i.e. that every element in $ \wh{D^*_i} $ is 
an evaluation at some point in $ S_i $. 
Let $ \pi\in \wh{D^*_i} $. Consider the kernel $ K_{i-1} $ of $R_{i-1}$, the restriction mapping 
from $ D^*_i $ into $ l^\iy(S_{i-1}) $. If $ \pi(K_{i-1})=\{0\} $, then we can consider $ \pi $ 
as being a representation of the quotient algebra $ D^*_i/K_{i-1} $. 
But the image of $ R_{i-1} $ is contained in $ D^*(S_{i-1}) $ and contains $ \F(A)\res{S_{i-1}} $ 
and so by assumption $ R_{i-1}(D^*_i)=\F(A)\res{S_{i-1}} $.  Hence $D^*_i/K_{i-1}\simeq\F(A){\res{S_{i-1}}}$ 
and therefore $ \pi  $ is an evaluation at a point in $ S_{i-1} $. 
If $ \pi(K_{i-1})\ne\{0\} $, then we look at the restriction of $ \pi $ to this ideal. 
The elements in $ K_{i-1} $ are operator fields defined on $ S_i $ which are norm continuous, 
which go to 0 at infinity and by condition 5., for any properly converging sequence 
$ \ol\ga\subset \GA_i $ with limit  outside $ \GA_i $, for every $ \va\in D^*_i $, we have that
\begin{eqnarray*}
\underset{k\to\iy}{\lim}\noop{\va(\ga_k)}&=&
\lim_{k\to\iy }\noop{\va(\ga_k)-\tilde\si_{\ol\ga,k}(\va\res{S_{i-1}})}=0.
\end{eqnarray*} This shows that $ K_{i-1}\subset C_0(\GA_i,\K(\H_i)) $. 
On the other hand the elements  of $ \F(A)\res{S_i}\cap K_{i-1} $ separate the points of $ \GA_i $ (see Proposition $2.11.2$ in \cite{Dix} for details). 
Therefore by the theorem of Stone-Weierstrass the algebras $ C_0(\GA_i,\K(\H_i)) $ and $ K_{i-1} $ coincide. 
This tells us that $ \pi $ is an evaluation at a point in $ \GA_i $, since the spectrum of the algebra $ C_0(\GA_i,\K(\H_i)) $ is homeomorphic to $ \GA_i $. 
 We conclude that $ \F(A)\res{S_i}=D^*_i $.
 \end{proof}
 \section{The list of the nilpotent Lie algebras of dimension $\leq5$ according to \cite{Nie}. }
There are 8 un-decomposable nilpotent non-abelian Lie algebras of dimension $\leq 5$: the Heisenberg Lie algebras 
of dimension 3 and 5, $\h_1 \text{ and } \h_2$, the thread-like Lie algebras of dimension 4 and 5,  $\f_4\text{ and } \f_5$, the step 2 Lie algebra $\g_{5,2}$, two step 3 Lie algebras, $\g_{5,3}$ and $\g_{5,4}$ and the step 4 Lie algebra $\g_{5,6}$.
 \begin{enumerate}
  \item The Heisenberg Lie algebras $\h_n, n\geq 1:$
  
    Let  $\h_n$ be the nilpotent Lie algebra of dimension 2n+1 spanned by the basis $$\B=\text{span}\left\{X_1,\cdots, X_n, Y_1,\cdots, Y_n,Z\right\}$$ equipped with the Lie bracket
 $$[X_i,Y_j]=\de_{i,j}Z, i,j=1,\cdots, n.$$

  \item The thread-like Lie algebras  $\f_n,n\geq 4:$
  
     Let  $\f_{n}$ be the nilpotent Lie algebra spanned by the basis $\B=\text{span}\{X_n,\cdots, X_1\}$ equipped with 
the Lie brackets:
\begin{eqnarray*}
 [X_n, X_j]=X_{j-1}, j=n-1,\cdots, 2.
 \end{eqnarray*}

  \item  the step $2$  Lie algebra $\g_{5,2}:$ 
  
   Let $\g_{5,2}$ be the nilpotent Lie algebra spanned by the basis $\B=\text{span}\{A,B,C,U,V\}$ equipped 
with the Lie brackets $$[C,A]=-U,\ [C,B]=V.$$
This is a semi-direct product of $\R$ with $\R^4$.
 \item  The step 3 Lie algebras $\g_{5,3}$ and $\g_{5,4}$:
  \begin{itemize}
 
 \item  Let $\g_{5,3}$ be the nilpotent Lie algebra spanned by the basis $\B=\text{span}\{A,B,C,U,V\}$ equipped with the Lie brackets
 $$[A,B]=U,\ [A,U]=V,\ [B,C]=V.$$
 This is  a semi-direct product of $\R$ with $\h_1\ti \R$.
   \item Let $\g_{5,4}$ be the nilpotent Lie algebra spanned by the basis
$\B=\{A,B,C,U,V\}$ equipped with the Lie brackets
\begin{eqnarray}\label{defofbracg54}
 \nn [A,B]= C,\ [A,C]= U,\  [B,C]= V.
\end{eqnarray} 
  \end{itemize}
  This is a semi-direct product of $\R$ with $\f_4\ti \R$.
\item The Step $4$ Lie algebra $\g_{5,6}:$

  Let $\g_{5,6}$ be the nilpotent Lie algebra spanned by the basis $\B=\{A,B,C,U,V\}$ equipped 
with the Lie brackets $$[A,B]= C,\ [A,C]= U,\  [A,U]= V,\ [B,C]=V.$$
This is a semi-direct product of $\R$ with $\f_4$.

 \end{enumerate}

 \section{The $C^*-$algebras of $H_n,n\geq 1$ and $F_n, n\geq 4$.}

 The $ C^* $-algebras of $ H_n $ and of $F_n $
have been realized  as  algebras of operator fields in \cite{Lud-Tur}.
It is shown there that the corresponding $C^* $-algebras have the norm-controlled dual limit property.

We give some details on the spectrum of these groups.
 \subsection{The groups $H_n$.}

The orbit space $\h_n^*/H_n$ consists of two layers:
\begin{enumerate}\label{}
\item the set of characters $\X=\GA^n_0=\sum_{j=1}^n \R X_j^*+\R Y_j^*$,
\item  the set $\GA^n_1$ of flat orbits:
\begin{eqnarray}\label{defofganHeis}
 \GA^n_1=\left\{\OM_\la=\la Z^*+Z^{*\perp}=\la Z^*+\sum_{j=1}^n \R X_j^*+\R Y_j^*, \la\in\R^*\right\}.
 \end{eqnarray}
 It is easy to see the relative topology of $\GA^n_1$ is Hausdorff, that $\OM_\la$ goes to infinity if $\la$ goes to infinity and that
 \begin{eqnarray*}
 \lim_{\la\to 0}\OM_\la=\GA^n_0.
 \end{eqnarray*}

The irreducible representation $\pi_\la$ associated to $\la\in\R^*$ acts on $L^2(\R^n)$ and for $F\in L^1(H_n)$ the operator $\pi_\la(F)$ is a kernel operator with kernel function $\F_\la$ given by
\begin{eqnarray*}
 F_\la(s,t):=\wh F^{2,3}(s-t, -\frac{\la}{2}(s+t),\la), s,t\in\R^n
\end{eqnarray*}
where
\begin{eqnarray*}
 \wh F^{2,3}(s,t,\la):=\int_{R^n\ti \R}F(s,u,z)e^{-2 i \pi(t\cdot u+\la z)}dudz.
 \end{eqnarray*}

 \end{enumerate}
 \subsection{The groups $F_n,\ n\geq 4$.} 
 
 For $n\geq4,$ let $\f_n$ be the $n-$dimensional real nilpotent Lie algebra with basis $\{X_1,\ldots,X_n\}$ and non-trivial Lie brackets
 $$[X_n,X_{n-1}]=X_{n-2},\ldots,[X_n,X_2]=X_1.$$
 For all  $\ell=\xi_1X_1^*+\ldots+\xi_{n-1}X_{n-1}^*$ and $t\in\R$ let 
 \begin{eqnarray*}
  t\cdot\ell&=&\Ad^*(\exp(tX_n))\xi\\
  &=& \left(0,\xi_{n-1}-t\xi_{n-2}+\ldots+\frac{1}{(n-2)!}(-t)^{n-2}\xi_1,\ldots,\xi_2-t\xi_1,\xi_1\right).
 \end{eqnarray*}
We define the function $\wh\ell$ on $\R$ by 
$$\wh\ell(t):=(t\cdot\xi)_{n-1}=\xi_{n-1}-t\xi_{n-2}+\ldots+\frac{1}{(n-2)!}(-t)^{n-2}\xi_1.$$ 
The mapping $\ell\to\wh\ell$ intertwines the $\Ad^*-$action and translation in the following way:
$$\wh{t\cdot\xi}(s)=(s\cdot(t\cdot\xi))_{n-1}=((s+t)\cdot\xi)_{n-1}=\wh\xi(s+t),\ \text{for } s,t\in\R.$$
This identifies the dual space $\f_n^*$ with the space $\P_{n-2}\oplus\R$, where $\P_{n-2}$ denotes the space  $\R[X,\leq{ n-2}]$ of 
real polynomials of degree $\leq n-2$. 
A co-adjoint orbit $\OM_\ell$ then corresponds to the translates
$\{\ell(t)P,t\in\R\}$ of some polynomial $P\in \R[X,\leq n-2]$. 

We can  parameterize the orbit space $\f_n^*/F_n$  with the sets
$$\underset{j=1}{\overset{n-2}{\bigcup}}\GA^n_j,$$
where $\GA^n_j:=\left\{\ell\in\f_n^*,\ \ell(X_k)=0,\ k=1,\ldots,j-1,j+1,\ \ell(X_j)\ne0\right\}.$

\begin{example}
 \begin{enumerate}
  \item For $n=4:$ The $4-$dimensional thread-like  algebra let  $F_4$ is the $4-$dimensional thread-like group which the multiplication
is given by 
$$(a,b,c,u)(a',b',c',u')=(a+a',b+b',c+c'-a'b,u+u'-a'c+\frac{a'^2b}{2}).$$
Let $\{A,B,C,U\}$ be the canonical basis of $\f_4$. For all  $\ell=\al A^*+\be B^*+\rh C^*+\mu U^*$ and $t\in\R$, the co-adjoint action is given by:
\begin{eqnarray}\label{Ad*g4}
\nn t\cdot\ell
&=&(\al,\be-\rh t+\mu\frac{t^2}{2},\rho-\mu t,\mu)
\end{eqnarray}
and 
$$\wh\ell(t):=\be-\rho t+\mu\frac{t^2}{2}.$$
We can parameterize the orbit space $\f_4^*/F_4$ in the following way. First we have a decomposition 
$$\f_4^*/F_4=\GA_2^4\cup\GA_1^4\cup\GA_0^4,$$
where 
\begin{eqnarray}\label{parmoforbspoff4}
 \nn\GA_2^4&=&\{\ell\in\f_4^*,\ \ell(U)\ne0,\ \ell(C)=0\}\\
 \GA_1^4&=&\{\ell\in\f_4^*,\ \ell(U)=0,\ \ell(C)\ne0,\ell(B)=0\}\\
 \nn\GA_0^4&=&\{\ell\in\f_4^*,\ \ell(U)=0,\ \ell(C)=0\}.
 \end{eqnarray}

\item For $n=5:$ The $5-$dimensional thread-like  algebra let  $F_5$ is the $5-$dimensional thread-like group which the multiplication
is given by 
$$(a,b,c,u,v)(a',b',c',u',v')=(a+a',b+b',c+c'-a'b,u+u'-a'c+\frac{a'^2b}{2},v+v'-a'u+\frac{a'^2c}{2}-\frac{a'^3b}{6}).$$
Let $\{A,B,C,U,V\}$ be the canonical basis of $\f_5$. For all  $\ell= \al A^*+\be B^*+\rh C^*+\mu U^*+\nu V^*$ and  $a\in\R,$  the co-adjoint action is given by:
\begin{eqnarray}\label{Ad*g5}
\nn t\cdot\ell&=&(\al,\be-\rho t+\mu\frac{t^2}{2}-\nu\frac{t^3}{6},\rho-\mu t+\nu\frac{t^2}{2},\mu-\nu t,\nu).
\end{eqnarray}
and  
$$\wh\ell(t):=\be-\rho t+\mu\frac{t^2}{2}-\nu\frac{t^3}{6}.$$
We can parameterize the orbit space $\f_5^*/F_5$ in the following way. First we have a decomposition 
$$\f_5^*/F_5=\GA_3^5\cup\GA_2^5\cup\GA_1^5\cup\GA_0^5,$$
where 
\begin{eqnarray*}
 \GA_3^5&=&\{\ell\in\f_5^*,\ \ell (V)\ne0,\ \ell(U)=0\},\\
\GA_2^5&=&\{\ell\in\f_5^*,\ \ell (V)=0,\ \ell(U)\ne0,\ \ell(C)=0\},\\
 \GA_1^5&=&\{\ell\in\f_5^*,\ \ell (V)=0,\ \ell(U)=0,\ \ell(C)\ne0,\ \ell(B)=0\},\\
 \GA_0^5&=&\{\ell\in\f_5^*,\ \ell (V)=0,\ \ell(U)=0,\ \ell(C)=0\}.
 \end{eqnarray*}
 \end{enumerate}

\end{example}

\section{ The $ C^* -$algebra of the group $G_{5,2}$.}
\subsection{}\label{intro}
The  description of the groups $G_{5,2}$  and of the three other ones of dimension 5,  of their  co-adjoint orbits and irreducible representations can be found 
in \cite{Nie}. 

We shall describe in this and in the following sections for our 4 remaining groups the limit sets of properly converging sequences $\ol O=(\pi_k)_k $ in their  dual spaces and we are explicitly constructing the mappings $\si_{\ol O,k} $. In this way it will turn out that the $C^* $-algebra of every connected Lie group of dimension $\leq 5 $ has norm controlled dual limits.

\subsection{}\label{gft}
 Recall that the Lie algebra of  $\g_{5,2}$ is  spanned by the basis $\B=\text{span}\{A,B,C,U,V\}$ equipped 
with the Lie brackets $$[C,A]=-U,\ [C,B]=V.$$
The Lie algebra $\g_{5,2}$ has a  two-dimensional centre $ \z=\textrm{span}\{U,V\} $. 
The group $G_{5,2}=\exp(\g_{5,2})$ can be realized as $ \R^{5} $ with the  Campbell-Baker-Hausdorff multiplication
\begin{eqnarray}\label{aaac3} 
\nn & &(a,b,c,u,v)(a',b',c',u',v')\\
\nn&=&(a+a',b+b',c+c',u+u'+\frac{1}{2}(ac'-a'c),v+v'+\frac{1}{2}(b'c-bc')
).
\end{eqnarray}
For all $(a,b,c,u,v)\in G_{5,2}$, $ (\al,\be,\rh,\mu,\nu)\in \g^*_{5,2} $ we obtain the following expression for $\Ad^*(a,b,c,u,v):$
\begin{eqnarray}\label{orgft}
\Ad^*(a,b,c,u,v)(x,y,t,\mu,\nu)=(x-\mu c,y+\nu c,t+\mu a-\nu b,\mu,\nu)
\end{eqnarray}

We give now a description of the co-adjoint orbits:
\begin{enumerate}\label{orftw}
 \item The generic elements $ (\al,\be,\rh,\mu,\nu) $  in $ \g^*_{5, 2} $ are  those for which $r^{2}_{\mu,\nu}=r^2=\mu^2+\nu^2\ne0$. 
 It follows from equation (\ref{orgft}) that a generic orbit $ \O $ is  determined by  3 parameters $ (\be,\mu,\nu) \in \R^3,r_{\mu,\nu}\ne 0 $:
 \begin{eqnarray*}\label{genftwdet}
\O=\O_{\be,\mu,\nu}=\left\{xA^*+y B^*+c C^*+\mu U^*+\nu V^*, c\in\R, \nu x+\mu y=\be r_{\mu,\nu} \right\}. 
\end{eqnarray*}
In particular, if $ \mu,\nu\ne 0 $, then the element 
\begin{eqnarray}\label{nunen}
 \frac{\be r_{\mu,\nu}}{\nu}A^*+\mu U^*+\nu V^* & 
\end{eqnarray}
is contained in $ \O_{\be,\mu,\nu} $ and if $ \nu=0 $ then the functional 
\begin{eqnarray}\label{nuisn}
\be \textrm{sign}(\mu)B^*+\mu U^*   
\end{eqnarray}
belongs to $ \O_{\be,\mu,0} $.

We take a new basis $ \B_{\mu,\nu} $ of $ \g_{5,2} $. For that,  letting $ \tilde
\mu:=\frac{\mu}{r_{\mu,\nu}},\tilde\nu:=\frac{\nu}{r_{\mu,\nu}} $ we put:
\begin{eqnarray}\label{basisp}
 \B_{\mu,\nu}&:=&\{A_{\mu,\nu}:=\tilde \mu A-\tilde \nu B, B_{\mu,\nu}:=\tilde\nu A+\tilde\mu B, C, U,V
\}. 
\end{eqnarray}
Let also
\begin{eqnarray*}
Z_{\mu,\nu}=\tilde\mu U+\tilde\nu V, T_{\mu,\nu}=\tilde \nu U-\tilde \mu V.
\end{eqnarray*}

Then:
\begin{eqnarray*}
[A_{\mu,\nu}, C] =Z_{\mu,\nu}, [B_{\mu,\nu},C] =T_{\mu,\nu}\\
\langle{\ell_{\be,\mu,\nu}},{[A_{\mu,\nu}, C]}\rangle=r_{\mu,\nu}.
\end{eqnarray*}

In this basis for any 
$R= x A_{\mu,\nu}+b B_{\mu,\nu}+y C +z Z_{\mu,\nu}+t T_{\mu,\nu} $, resp
$ R' =x' A_{\mu,\nu}+b' B_{\mu,\nu}+y' C +z' Z_{\mu,\nu}+t' T_{\mu,\nu} $ in $ \g $ we obtain the multiplication:
\begin{eqnarray}\label{hausmulhei}
\nn R\cdot R'&=& (x+x')A_{\mu,\nu}+(b+b') B_{\mu,\nu}+(y+y') C+\\
 &+& \left(z+z'+\frac{1}{2}(xy'-x' y) \right)Z_{\mu,\nu}+\left(t+t'+\frac{1}{2}(by'-b'y)\right) T_{\mu,\nu}.
\end{eqnarray}
We see that the vectors $ A_{\mu,\nu}, B_{\mu,\nu},Z_{\mu,\nu} $ span the three dimensional Heisenberg Lie algebra, that $ B_{\mu,\nu} $ is contained in the stabilizer of the linear form $ \ell_{\be,\mu,\nu} $ and that $ T_{\mu,\nu} $ is contained in the kernel of the restriction of $ \ell_{\be,\mu,\nu} $ to the centre of $ \g_{\mu,\nu} $. 
 In the
dual basis $ \B^*_{\mu,\nu} $ of the basis $\B_{\mu,\nu}$ the orbit $ \O_{\be,\mu,\nu} $ (for some $\be\in\R$) of the
element 
\begin{eqnarray*}
\ell_{\be, \mu,\nu}=\be B_{\mu,\nu}^*+\mu U^*+\nu V^*, 
\end{eqnarray*}
is given by:
\begin{eqnarray}\label{seclayortwfi}
 \O_{\be,\mu,\nu}&=&\{a A_{\mu,\nu}^* +\be B_{\mu,\nu}^*+c C^*+\mu U^*+\nu V^*, a, c \in\R \}. 
\end{eqnarray}
 The stabilizer of $\ell_{\be,\mu,\nu}$ is the set\\ 
$$\g_{5,2}(\ell_{\be,\mu,\nu})=span\{
B_{\mu,\nu},U,V\}.$$
We can take as polarization at $ \ell_{\be,\mu,\nu} $ the sub-algebra
\begin{eqnarray*}
\p:=\textrm{span}\{A,B,U,V\},\ P:=\exp(\p).
\end{eqnarray*}
We denote by $ \GA^{5,2}_1 $ the orbit space of this layer and we parametrize it by
\begin{eqnarray}\label{ga1}
\GA^{5,2}_1:=\{\ell_{\be,\mu,\nu}\equiv(\be,\mu,\nu), \be\in\R, (\mu,\nu)\in \R^2, r^2=\mu^2+\nu^2\ne 0\}.
\end{eqnarray}
\item The second layer,  denoted  by 
\begin{eqnarray}\label{ga0}
\GA^{5,2}_0=(\g_{5,2}^*/G_{5,2})_{char}\simeq \R^3 
\end{eqnarray}
is the
collection of all characters $\ell_{\al,\be,\rho}=\al A^*+\be B^*+\rho C^*,\
\al,\be,\rho\in\R$. Their orbits are the one point sets
$\{\ell_{\al,\be,\rho}\}$ 
\end{enumerate}
\begin{theorem}\label{topdualft}
$  $
\begin{enumerate}\label{}
\item On the set $ \GA^{5,2}_1 $ the dual topology is Hausdorff.
\item  Let $ \ol O=(\O_{\be_k,\mu_k,\nu_k})_k $ be  a sequence such that $ \lim_k r_k=0 $. Then this sequence has  a converging sub-sequence if and only if $ \liminf_k \vert\be_k\vert $ is finite. 
If $ \ol O $ is properly converging, then, passing to a sub-sequence (also denoted by the same 
symbol for simplicity of notations) we can assume that $ \lim_k \be_k=\be $ exists, that the sequences 
of vectors $(A_k^*=A_{\mu_k,\nu_k}^*)_k, \textrm{ resp. } (B_k^*=B^*_{\mu_k,\nu_k})_k $  converges to a
$ A^*_\iy, \textrm{ resp. }B_\iy^*  $ in $ \g^*_{5,2} $ and then 
\begin{eqnarray*}
L(\ol O)=\R A^*_\iy+\be_\iy B^*_\iy+\R C^*\subset \GA^{5,2}_0.
\end{eqnarray*}

 \end{enumerate}

 \end{theorem}
\begin{proof}
\begin{enumerate}\label{}
\item The point 1) is evident. 
\item  If $ \liminf_k \val{\be_k} $ exists in $ \R $, then we take a sub-sequence (indexed  also by $ (\be_k)_k  $ 
for simplicity of notation) such that $ \lim_k  \be_k=\be $ exists in $ \R $ and such that the sequences of 
vectors $(A_k^*=A_{\mu_k,\nu_k}^*)_k, \textrm{ resp. } (B_k^*=B^*_{\mu_k,\nu_k})_k $  converges to a
$ A^*_\iy, \textrm{ resp. }B_\iy^*  $ in $ \g_{5,2}^* $. It follows then from the description (\ref{seclayortwfi}) of the coadjoint orbits that 
the limit set $ L(\ol O) $   of the sub-sequence  is the set described in the theorem. 
If $ \liminf_k \val{\be_k}=+\iy $, then $ \lim_k \{\val{\langle{O_k},{B_k}\rangle}\}=\{\val{\be_k}\}=+\iy $. Hence $ \ol O $ goes to infinity 

 \end{enumerate}

 \end{proof}

\subsection{ The unitary dual of ${G_{5,2}}$.}
The spectrum of the group $G_{5,2} $ can be identified by Kirillov's orbit theory with the orbit space $ \g^*_{5 , 2} /G_{5,2}$ = $\wh{G_{5,2}}:= \GA^{5,2}_1\cup\GA^{5,2}_0 $. 
For every $(\be,\mu,\nu)\in\R\times(\R\times\R)^*$, we take the irreducible  representation 
$ \pi_{\be,\mu,\nu} =\ind{P}{G_{5,2}}\ch_{\be,\mu,\nu} $ which is associated to $ O_{\be,\mu,\nu} $. This representation acts on the 
Hilbert space $L^2(\R)$ and  is given by the formula 
\begin{eqnarray*}
 & &\pi_{\be,\mu,\nu}(x,b,y,z,t)\xi(s)\\
 &=&e^{-2i\pi(\be b +r_{\mu,\nu}z+r_{\mu,\nu}x s-\frac{r_{\mu,\nu}}{2}xy)}\xi(s-y),\ s\in\R,\ \xi\in L^2(\R),\
(x,b,y,z,t)\in G_{5,2},
\end{eqnarray*}
using the coordinates coming from the basis 
$$ \B^{\mu,\nu}:=\{A_{\mu,\nu},B_{\mu,\nu},C,Z_{\mu,\nu},T_{\mu,\nu}\} $$ 
and formula (\ref{hausmulhei}).
Since $G_{5,2}$ is nilpotent, by Lie's theorem every irreducible finite
dimensional representation of $G_{5,2}$ is one-dimensional. Any
one-dimensional representation is a unitary character $\chi_{\al,\be,\rho},\
(\al,\be,\rho)\in\R^3$, of $G_{5,2}$ which is given by
$$\chi_{\al,\be,\rho}(a,b,c,u,v)=e^{-2i\pi(\al a+\be b+\rho c)},\
(a,b,c,u,v)\in G_{5,2}.$$ 
For $F\in L^1(G_{5,2})$, let 
$$\widehat
F(\al,\be,\rho):=\chi_{\al,\be,\rho}(F)=\int_{G_{5,2}}F(a,b,c,0,0)e^{
-2\pi i(\al a+\be b+\rho c)}dadbdc,\ \al,\be,\rho\in\R, $$
and $$\Vert
F\Vert_{\infty,0}:=\underset{\al,\be,\rho\in\R}{\sup}\vert\chi_{\al,\be,\rho}
(F)\vert=\Vert\hat F\Vert_{\infty}.$$
\begin{definition}
Define for $a\in C^*(G_{5,2})$ its  Fourier transform $\F(a)\in L^{\iy}(\wh{G_{5,2}})$
by
$$\F(a)(\be,\mu,\nu):=\pi_{\be,\mu,\nu}(a)\in
\B(L^2(\R)),\
(\be,\mu,\nu)\in\GA^{5,2}_1$$
and 
$$\F(a)(\al,\be,\rho):=\chi_{\al,\be,\rho}(a), \ \al,\be,\rho\in\R. $$
\end{definition}
\begin{definition}
 For all $F\in L^1(G_{5,2})$, and $ (\be,\mu,\nu)\in \GA^{5,2}_1 $ the   operator
$\pi_{\be,\mu,\nu}(F)$ is a kernel operator with kernel function $ F_{\be,\mu,\nu} $ given by:
\begin{eqnarray*}\label{kernfuncbemunu}
F_{\be,\mu,\nu}(s,x)&=&\wh F^{{\mu,\nu}}(\frac{r_{\mu,\nu}}{2}(s+x)A^*_{\mu,\nu}+\be B^*_{\mu,\nu}+(s-y)C^*+r_{\mu,\nu}Z_{\mu,\nu}^*) 
\end{eqnarray*}
in the coordinates coming from the basis $ \B^{\mu,\nu} $. Here the symbol $ \wh F^{\mu,\nu} $ denotes the function
\begin{eqnarray*}
\wh F^{\mu,\nu}(s,q):=\int_{P}F(\exp{(s C)p})e^{-2\pi i \langle{q},{\log p}\rangle}dp, s\in\R,q\in\p^*.
\end{eqnarray*}
Indeed, for 
 for $\xi\in L^2(\R)$  and $s\in\R$ we have 
\begin{eqnarray}\label{pialmuga}
 \nn
\pi_{\be,\mu,\nu}(F)\xi(s)&=&\int_{\R^{5}}F(x,b,y,z,t)\pi_{\be,\mu,\nu} (x,b,y,z,t)\xi(s)dxdbdydzdt\\
\nn&=& \int_{\R^5}F(x,b,y,z,t)e^{-2i\pi(\be b +r_{\mu,\nu}z+r_{\mu,\nu}x s-\frac{r_{\mu,\nu}}{2}xy)}\xi(s-y)dxdbdydzdt\\
\nn&=& \int_\R\left(\int_{\R^4}F(x,b,s-y,z,t)e^{-2i\pi(\be b +r_{\mu,\nu} z+\frac{r_{\mu,\nu}(s+y)x}{2})}\xi(y)dxdbdzdt\right)dy\\
&=&\int_\R\wh F^{{\mu,\nu}}(\frac{r_{\mu,\nu}}{2}(s+x)A^*_{\mu,\nu}+\be B^*_{\mu,\nu}+(s-y)C^*+r_{\mu,\nu}Z_{\mu,\nu}^*)\xi(y)dy.
\end{eqnarray}          
\end{definition}
The following proposition is a consequence of Formula (\ref{pialmuga}).
\begin{proposition}
 For any $a\in C^*(G_{5,2})$ and $(\be,\mu,\nu)\in\GA^{5,2}_1$,  the operator $\pi_{\be,\mu,\nu}(a)$ is compact, the
mapping $\GA^{5,2}_1\to
B(L^2(\R)):(\be,\mu,\nu)\mapsto\pi_{\be,\mu,\nu}(a)$ is norm continuous in $(\be,\mu,\nu)$ and tending to
$0$ for $r_{\mu,\nu}$ going to infinity. 
\end{proposition}
\begin{definition}\label{}
\rm     Let as before for $ \mu^2+\nu^2\ne 0 $, $ A_{\mu,\nu}:=\tilde\mu A-\tilde\nu B $,  $
B_{\mu,\nu}:=\tilde \nu A+\tilde\mu B $.
 Choose a Schwartz-function
$\eta$ in $\S(\R)$ with $L^2-$norm equal to $1$. For $(\al,\rh)\in\R^2$ we
define the function $\eta_{\mu,\nu}(\al,\rh)$ by
\begin{eqnarray}\label{etade}
  & &\eta_{\mu,\nu}(\al,\rh)(s):= r_{\mu,\nu}^{\frac{1}{4}}
e^{2i\pi s\rh}\eta\left( r_{\mu,\nu}^{\frac{1}{2}}(s+\frac{\al}{r_{\mu,\nu}})\right),\ s\in\R.
\end{eqnarray}

 \end{definition}

\subsection{ A $C^*-$condition.}
The $ C^* $-conditions for the group $ G_{5,2} $ can be copied from the corresponding conditions for the Heisenberg groups (see \cite{Lud-Tur}).
\begin{lemma}
 Let $\xi\in\S(\R)$. Then,
$$\xi=\frac{1}{ r_{\mu,\nu}}\int_{\R^2}\langle\xi,\eta_{\mu,\nu}(\al,\rh)\rangle\eta_{\mu,\nu}(\al,\rh)d\al d\rh.$$ 
\end{lemma}
\begin{proof}
 The proof is the same as the proof of  Lemma $2.8$ in [Lud,Tur].
\end{proof}
\begin{definition}
 $1.$\ For all $(\al,\rh)\in\R^2$ and $k\in\N$, let $P_{{\mu,\nu}(\al,\rh)}$ be the
orthogonal projection onto  the one dimensional subspace
$\C\eta_{\mu,\nu}(\al,\rh)$.\\
$2.$ Define for $h\in C_0(\R^3)$  the linear operator
\begin{eqnarray}\label{nuP}
 \sigma_{\be,\mu,\nu}(h):=\frac{1}{ r_{\mu,\nu}}\int_{\R^2}
h(\al A_{\mu,\nu}^*+\be B_{\mu,\nu}^*+ \rh C^*)P_{\mu,\nu(\al,\rh)}d\al d\rho
\end{eqnarray}
\end{definition}
\begin{proposition}(see Proposition 2.11 in \cite{Lud-Tur})
 \begin{enumerate}
  \item For every ${(\be,\mu,\nu)}\in\GA^{5,2}_1$ and $h\in\S(\R^2)$ the integral $(\ref{nuP})$ converges in operator norm.
\item   $ \si_{\be,\mu,\nu}(h) $ is compact and
$\no{\si_{\be,\mu,\nu}(h)}_{op}\leq\no{h}_{\iy}$.
\item  The mapping $ \si_{\be,\mu,\nu}: C_0(\R^{2})\to \F_1 $
is involutive, i.e. $ \si_{\be,\mu,\nu}(h^*)=\si_{\be,\mu,\nu}(h)^* , h\in C^*(\R^{2})$, where by
$\si_{\be,\mu,\nu}$ we denote also  the extension of $\si_{\be,\mu,\nu}$ to $C_0(\R^{2})$.
 \end{enumerate}

\end{proposition}

\begin{theorem}
 Let $a\in C^*(G_{5,2})$ and let $\va$ be the operator field
$\va=\F(a)$. Then the function $ \va(\be):(\al,\rh)\to \F(a)(\al,\rh) $ is contained $ C_0(\R^2) $. Let $ \ol\ga=(\ga_k=(\be_k,\mu_k,\nu_k))_k $ be a properly converging sequence in $ \GA^{5,2}_1 $ having its limit set $ L(\ol\ga)=\R A_\iy^* +\be B_\iy^* +\R C^* $ in $ \GA^{5,2}_0 $.  
Then
\begin{eqnarray}
 \underset{k\to\infty}{\lim}\Vert
\va(\ga_k)-\sigma_{\ga_k}(\va(\be))\Vert_{op}=0.
\end{eqnarray}
\end{theorem}
\begin{proof}
 The proof is the same as that of  Theorem $2.12$ in [Lud-Tur]. 
\end{proof}

 \section{The $C^*-$algebra of the group $G_{5,3}.$}
 
 Recall that the Lie algebra of    $\g_{5,3}$ is  spanned by the basis $\B=\text{span}\{A,B,C,U,V\}$ equipped with the Lie brackets
 $$[A,B]=U,\ [A,U]=V,\ [B,C]=V.$$

 This Lie algebra has  a one-dimensional centre $\z=\R V.$ The group $G_{5,3}$ can be realized as $\R^5$ with the Campbell-Baker-Hausdorff 
 multiplication 
 \begin{eqnarray}\label{multiG53}
  \nn& &(a,b,c,u,v)(a',b',c',u',v')\\
  \nn&=&(a+a',b+b',c+c',u+u'-a'b,v+v'-a'u+\frac{a'^2b}{2}-\frac{b'c}{2}+\frac{bc'}{2}).  
 \end{eqnarray}
For all $(a,b,c,u,v)\in G_{5,3},\ (\al,\be,\rh,\mu,\nu)\in\g_{5,3}$ we obtain the following expression for $\Ad^*(a,b,c,u,v):$
\begin{eqnarray}\label{defAd*G53}
 \nn& &\Ad^*(a,b,c,u,v)(\al,\be,\rh,\mu,\nu)\\
 &=&(\al-\mu b-\nu u-\nu\frac{ab}{2},\be+\mu a-\nu c+\nu \frac{a^2}{2},\rh+\nu b,\mu+a \nu,\nu).
\end{eqnarray}
We give now a description of the co-adjoint orbits:
\begin{enumerate}
 \item The generic orbits: They have a non-zero value $\nu$ on the central element $V.$ It follows from (\ref{defAd*G53}) that we can 
 characterize such an orbit $\O$ by $\nu\in\R.$ There exists in each generic orbit $\O_\nu$ a unique element $\ell_\nu$ which 
 is zero on the vectors $A,B,C,U,\ i.e.\ \ell_\nu=\nu V^*$ and so 
 $$\O_\nu=\{(a,b,c,u,\nu),\ (a,b,c,u)\in\R^4\}.$$
 We denote by $\GA_2^{5,3}=(\g_{5,3}^*/G_{5,3})_{\text{gen}}$ this family of generic co-adjoint orbits, parameterized by the set 
 $\GA_2^{5,3}:=\{\ell_\nu\equiv\nu;\ \nu\in\R\}.$  
 \item The second layer is given by the set of linear functionals, which are $0$ on $V,$ but not $0$ on $U$, we can characterize such 
 an orbit $\O$ by the pair $(\rh,\mu)\in\R^2$ and so 
 $$\O_{\rh,\mu}=\{(a,b,\rh,\mu,0),\ (a,b)\in\R^2\}.$$
 \item The last layer, denoted by $\GA_0^{5,3}=(\g_{5,3}/G_{5,3})_{\text{char}}\simeq\R^3$ is the collection of all characters 
 $\ell_{\al,\be,\rh}=\al A^*+\be B^*+\rh C^*,\ \al,\be,\rh\in\R.$ Their orbits are the point set $\{\ell_{\al,\be,\rh}\}.$
\end{enumerate}
\begin{theorem}\label{topodualG53}
\begin{enumerate}
 \item On the set $\GA_2^{5,3}$ $($ resp on the set $\GA_1^{5,3}$ $)$ the dual topology is Hausdorff.
 \item Let $\ol\O=(\O_{\nu_k})_k\subset\GA_2^{5,3}$ be a sequence, such that $\underset{k\to\iy}{\lim}\nu_k=0.$ Then $\ol\O$ 
is properly converging and  $L(\ol\O)=\GA_1^{5,3}\cup\GA_0^{5,3}.$
\item Let $\ol\O=(\O_{\rh_k,\mu_k})_k$ be a sequence such that $\underset{k\to\iy}{\lim}\mu_k=0.$ If $\ol\O$ has a limit then $\rh:=\lim_k\rh_k$ 
exists in $\R.$ Conversely, if $\lim_k\rh_k=\rh$ exists,  then the sequence $\ol\O$ converges and $L(\ol\O)=\R A^*+\R B^*+\rh C^*.$
\end{enumerate}
\end{theorem}
\begin{proof}
 The proof is straight forward.
\end{proof}
\subsection{The Fourier transform for $C^*(G_{5,3}).$}
The spectrum of the group $G_{5,3}$ can be identified by Kirillov's orbit theory with the orbit space 
$\g_{5,3}^*/G_{5,3}=\wh {G_{5,3}}=\GA_2^{5,3}\cup\GA_1^{5,3}\cup\GA_0^{5,3}.$ 
\begin{enumerate}
 \item Let $\ell\in\GA_2^{5,3},$ its orbit $\O_\ell$ is of dimension $4$. A polarization at $\ell_\nu=\nu V^*$ is given by 
 $\p=\p_\nu:=\text{span}\{C,U,V\}.$ We realize then $\pi_\nu:=\ind{P_\nu}{G_{5,3}}\chi_\nu.$ The Hilbert space $L^2(G_{5,3}/P_\nu,\chi_\nu)$ is in fact isometric 
 to $L^2(\R^2),$ let $E:\R^2\to G_{5,3},\ E(a,b):=\exp(aA)\exp(bB)$ and $S =\exp(\R A)\exp(\R B)=E(\R^2).$ Then $G_{5,3}=S.P_\nu$ 
 as topological product and the mapping $\la:L^2(G_{5,3}/P_\nu,\chi_\nu)\to L^2(\R^2)$ defined by $\la\xi(t):=\xi(E(t)),\ t\in\R^2,$
  is unitary. Let us compute the operator $\pi_\ell(F)$ for $F\in C^*(G_{5,3})$ explicitly. For $\xi\in L^2(\R^2),\ s\in S,\ p\in P_\nu$ we have  
  \begin{eqnarray*}
   \pi_\ell(F)\xi(t)=\int_{G_{5,3}/P_\nu}\xi(s)\left(\int_{P_\nu}F(ts\inv p)e^{-2i\pi\langle s.\ell,p\rangle}dp\right)ds.
  \end{eqnarray*}
For $t=E(a,b)$ and $s=E(a',b')$ we get 
\begin{eqnarray*}\label{defopgenG53}
 \pi_\ell(F)\xi(a',b')&=&\int_{\R^2}\wh F^{P_\nu}(a'-a,b'-b,E(a,b).\ell\res{\p_\nu})e^{\frac{-2i\pi\nu a^2(b'-b)}{2}}\xi(a,b)dadb.
 \end{eqnarray*}
\item Let $\ell=\ell_{\rh,\mu}\in\GA_1^{5,3}.$ A polarization at $\ell$ is given by $\p_{\rh,\mu}=\text{span}\{B,C,U,V\}.$ 
We take $\pi_{\rh,\mu}:=\ind{P_{\rh,\mu}}{G_{5,3}}\chi_{\rh,\mu}.$ This representation acts on 
the Hilbert space $L^2(G_{5,3}/P_{\rh,\mu},\chi_{\rh,\mu})\simeq L^2(\R)$ and for $F\in L^1(G_{5,3}),\ \xi\in L^2(\R),$ we have:     
\begin{eqnarray*}
 \pi_{\rh,\mu}(F)\xi(a')=\int_{\R}\wh F^{P_{\rh,\mu}}(a'-a,a\cdot p_{\rh,\mu}) da,\ \text{ where } p_{\rh,\mu}={\ell_{\rh,\mu}}\res{\p_{\rh,\mu}}.
\end{eqnarray*}
\item Any one-dimensional representation is a unitary character $\chi_{\al,\be,\rh},\ (\al,\be,\rh)\in\R^3,$ of $G_{5,3}$ which is given 
by 
$$\chi_{\al,\be,\rh}(a,b,c,u,v)=e^{-2i\pi(\al a+\be b+\rh c)},\ (a,b,c,u,v)\in G_{5,3}.$$
For $F\in L^1(G_{5,3}),$ let 
$$\wh F(\al,\be,\rh):=\chi_{\al,\be,\rh}(F)=\int_{G_{5,3}}F(a,b,c,0,0)e^{-2i\pi(\al a+\be b+\rh c)}dadbdc,\ \al,\be,\rh\in\R.$$
\end{enumerate}
\begin{definition}
 Define for $a\in C^*(G_{5,3})$ its Fourier transform $\F(a)\in l^\iy(\wh{G_{5,3}})$ by 
 \begin{eqnarray}\label{fouriertranG53}
\nn&&\hat
 a(\nu)=\F(a)(\nu):=\pi_{\nu}(a)\in\K(\l2{\R^{2}}),\ \nu\in\GA_2^{5,3};
\\
\nn &&\hat
a(\rh,\mu)=\F(a)(\rh,\mu):=\pi_{\rh,\mu}(a)
\in\K(\l2\R),\ (\rh,\mu)\in\GA_1^{5,3};\\
\nn&&\hat a(\al,\be,\rh)=\F(a)(\al,\be,\rh):=\chi_{\al,\be,\rh}(a)\in C^*(\R^3).
\end{eqnarray}
\end{definition}
\begin{proposition}
  For any $a\in C^*(G_{5,3})$ and $\nu\in\GA_2^{5,3} (\text{resp } (\rh,\mu)\in\GA_1^{5,3}),$ the operator $\pi_\nu(a)$ 
  (resp the operator  $\pi_{\rh,\mu}$) is compact,  
  the mapping $\GA_2^{5,3}\to \K(L^2(\R^2)):\nu\mapsto\pi_\nu(a)$ (resp  the mapping $\GA_1^{5,3}\to\K(L^2(\R)):(\rh,\mu)\mapsto\pi_{\rh,\mu}(a)$)
  is norm continuous in $\nu$ (resp in $(\rh,\mu)$) and tending to $0$ for $\nu$ going to infinity (resp for $\rh$ or $\mu$ going to infinity).

\end{proposition}

\subsection{The changing of layers condition.}
\subsubsection*{$\bullet$ Passing from $\GA_2^{5,3}$ to $\GA_1^{5,3}\cup\GA_0^{5,3}$.}
Let $\ol\O=(\O_{\ell_k})_k\subset\GA_2^{5,3}$ be a properly converging sequence where $\ell_k=(0,0,0,0,\nu_k),\ k\in\N$ such that 
$\lim_k\nu_k=0.$ Let $p_k:=(\ell_k)\res{\p}.$ By Theorem \ref{topodualG53} the restriction of the limit set $L(\ol\O)$ to $\p$ 
is the closed set $L=L(\O)\res\p=\{(\rh,\mu,0),\ \rh\in\R, \mu\in\R\}.$
\begin{definition}
 For $k\in\N$ let:
\begin{eqnarray}\label{defset}
\nn\ve_k&:=&\val{\nu_k}^{\frac{1}{2}},\\
\nn I_{i,j}^k&:=&\left\{(c,u,\nu_k)\in p^*;
i\ve_k^{\frac 1 4}\leq{c}< i\ve_k^{\frac 1 4}+\ve_k^{\frac 1 2}
\text{ and }j\ve_k\leq u<j\ve_k+\ve_k\right\},\\
\nn U_{i,j}^k&:=&\left\{(x,y)\in\R^2;(xA+yB)\cdot p_k\in I_{i,j}^k\right\},\ j\in\Z^*.
\end{eqnarray}
Finally:
\begin{eqnarray*}
 U^k&:=&\underset{i,j\in\Z^*}{\bigcup} U_{i,j}^k.
\end{eqnarray*}
Let also for $k\in\N,\ i,j\in\Z:$
\begin{eqnarray}\label{defseqG56}
\nn& & x_{j}^k :=\frac{j\ve_k}{\nu_k},\ y_{i}^k:=\frac{i\ve_k^{\frac 1 4}}{\nu_k},
\ g_{i,j}^k=x_{j}^kA+y_{i,j}^kB.
\end{eqnarray}
Let for $i,j\in\Z^*,\ k\in\N^*:$
\begin{eqnarray*}
& &p_{i,j}^k:=(i\ve_k^{\frac{1}{4}},j\ve_k,0).
\end{eqnarray*}
\end{definition}
An easy computation gives:
\begin{eqnarray}\label{calcg.pG56}
 \nn& & g_{i,j}^k.p_k=(i\ve_k^{\frac{1}{4}},j\ve_k,\nu_k)=p_{i,j}^k+(0,0,\nu_k).
\end{eqnarray}
\begin{proposition}\label{suppcompG53}$ $
  Let $K$ be a compact subset, for $k$ large enough we have that 
  $$K U_{i,j}^k\subset \underset{i',j'=-1}{\overset{1}{\bigcup}}U_{i'+i,j'+j}^k=:V_{i,j}^k.$$
\end{proposition}
\begin{proof}$ $
  We can suppose that $K P$ is contained in $[-M,M]^2P$ for some $M>0.$
For $r=(u,v)\in K \subset  G_{5,3}/P$ and $s=(x,y)\in U^k$ we have that 
\begin{eqnarray*}
 (rs).p_k&=&(\nu_kv+\nu_ky,\nu_kx+\nu_k u,\nu_k)
\end{eqnarray*}
and 
\begin{eqnarray*}
 \nn& & (x,y)\in U_{i,j}^k\\
\nn &\Leftrightarrow&(x,y)\cdot p_k\in I^{k}_{i,j}\\
\nn &\Rightarrow&
\left\{
\begin{array}{c}
j\ve_k\leq \nu_kx<j\ve_k+\ve_k,\\
i\ve_k^{\frac 1 4} \leq\nu_ky< i\ve_k^{\frac 1 4}+\ve_k^{\frac 1 2},
\end{array}
\right.\\
\nn &\Rightarrow&
\left\{
\begin{array}{c}
(j-1)\ve_k\leq \nu_kx+\nu_ku<(j+1)\ve_k+\ve_k,\\
(i-1)\ve_k^{\frac 1 2} \leq\nu_ky+\nu_kv< (i+1)\ve_k^{\frac 1 4}+\ve_k^{\frac 1 2}.
\end{array}
\right.
 \end{eqnarray*}
It follows that $\K U_{i,j}^k\subset \underset{i',j'=-1}{\overset{1}{\bigcup}}U_{i'+i,j'+j}^k.$
\end{proof}
\begin{definition}
 For $k\in\N^*$
   Let $$R^k=\left[-\frac{\ve_k}{\val{\nu_k}},\frac{\ve_k}{\val{\nu_k}}\right]\times
  \left[-\frac{\ve_k^{\frac{1}{2}}}{\val{\nu_k}},\frac{\ve_k^{\frac{1}{2}}}{\val{\nu_k}}\right].$$
\end{definition}
\begin{lemma}\label{lemma1G53}
 For $k\in\N^*$ large enough,  for any $i,j\in\Z^*$ we have  the set $U_{i,j}^k$ is contained in $ R^k+g_{i,j}^k$. 
\end{lemma}
\begin{proof}
 Let $s=(x,y)\in U_{i,j}^k.$ Then:
 \begin{eqnarray*}
  & &(x,y)\cdot p_k\in I_{i,j}^k\\
  &\Longleftrightarrow& \begin{cases}
j\ve_k\leq \nu_kx<j\ve_k+\ve_k\Rightarrow\val{x-x_j^k}\leq\frac{\ve_k}{\val{\nu_k}}\Rightarrow 
x\in\left[-\frac{\ve_k}{\val{\nu_k}},\frac{\ve_k}{\val{\nu_k}}\right]+x_j^k,\\
i\ve_k^{\frac 1 4} \leq\nu_ky< i\ve_k^{\frac 1 4}+\ve_k^{\frac 1 2}
\Rightarrow\val{y-y_{i}^k}\leq \frac{\ve_k^{\frac 1 2}}{\val{\nu_k}}\Rightarrow 
y\in\left[-\frac{\ve_k^{\frac{1}{2}}}{\val{\nu_k}},\frac{\ve_k^{\frac{1}{2}}}{\val{\nu_k}}\right]+y_{i}^k.
                         \end{cases}\\
\nn&\Longrightarrow& s\in R^k+g_{i,j}^k,
\end{eqnarray*}

\end{proof}
\begin{lemma}\label{lemma2G53}
  For $k\in\N^*$ large enough, for $i,j\in\Z^*$ and any 
 $(x,y)\in U_{i,j}^k$ we have that 
 $$\no{(xA+yB)\cdot p_k-((xA+yB)\cdot (g_{i,j}^k)\inv)\cdot p_{i,j}^k}\leq3\ve_k^{\frac 1 2}.$$
\end{lemma}
\begin{proof}
 For $(x,y)\in U_{i,j}^k$ we have that $ (x,y)=(x'+x^k_{j},y'+y^k_{i,j}) $ where $ \val{\nu_kx'}\leq \ve_k $ 
 and $ \val{\nu_k y'}\leq  \ve_k^{{\frac{1}{2}}}$. Therefore
\begin{eqnarray}\label{}
\nn &&
\no{((x'+x^k_{j})A+(y'+y^k_{i,j})B)p_k-(x'A+y'B)\cdot p^{k}_{i,j}}\\
\nn &=&
\no{(\nu_ky'+\nu_ky_{i}^k,\nu_kx_j^k+\nu_kx',\nu_k)
-(i\ve_k^{\frac 1 4},j\ve_k,0)}\\ 
\nn &=&
\no{(\nu_ky'+i\ve_k^{\frac 1 4},j\ve_k+\nu_kx',\nu_k)
-(i\ve_k^{\frac 1 4},j\ve_k,0)}\\ 
\nn &=&
\no{(\nu_ky',\nu_kx',\nu_k)}\\
\nn&=& \val{\nu_ky'}+\val{\nu_kx'}+\val{\nu_k}\\
\nn&\leq&\ve_k^{\frac 1 2}+\ve_k+\val{\nu_k}\leq3\ve_k^{\frac 1 2}.
\end{eqnarray}
\end{proof}
\begin{definition}\label{piispipG53}
\rm   Let for $ \be,\ \rh,\ \mu\in\R $,  let 
$$ \ell_{\be,\rh,\mu}=\be B^*+\rh C^*+\mu U^*,\ \ell_{\be,\mu}=\be B^*+\mu U^*,\ \text{and }\ell_{\rh,\mu}=\rh C^*+\mu U^*\in\g_{5,3}^*. $$
 The sub-algebra 
$ \p:=\p_{{\be,\mu}}=\textrm{ span}\{B,\p\} $ is a polarization  at $\ell_{\be,\mu} $ and at $\ell_{\be,\rh, \mu} $, which
gives us  the equivalent 
representations  $ \pi_{\be,\mu}=\ind {P}{G_{5,3}}
\ch_{\ell_{\be,\mu}}\in\wh G_{5,3} $ and $ \pi_{\be,\rh,\mu}=\ind
{P}{G_{5,3}}
{\ch_{\ell_{\be,\rh,\mu}}} $. Let $ u_{\be,\rh,\mu} $ be the unitary operator which gives the equivalence between both representations. 
We take the direct integral representation 
\begin{eqnarray}\label{integralrepG53}
 \ta_{\rh,\mu} :=\left(\int_\R^\oplus
\pi_{\ell_{\be,\rh,\mu}}d\be,\int_\R^\oplus
L^2(G_{5,3}/P,\ch_{\ell_{\be,\rh,\mu}})d\be\right). 
\end{eqnarray}
This representation $ \ta_{\rh,\mu}  $ is in fact equivalent to the representation 
 $ \si_{\rh,\mu}:=\ind {P} {G_{5,3}}\ch_{\ell_{\rh,\mu}} $ and a unitary  intertwining $ U_{\rh,\mu} $ operator is given by:
\begin{eqnarray}\label{}
 \nn U_{\rh,\mu}:& L^{2}(G_{5,3}/P,\ch_{\rh,\mu})\mapsto\int_\R^\oplus
L^{2}(G_{5,3}/P,\ch_{\ell_{\rh,\mu}})d\be,\\
\nn  
U_{\rh,\mu}(\xi)(\be)(g):&=\int_\R \xi(g\exp{( s B)})e^{-2i\pi s\be}ds,g\in G_{5,3},\be\in\R.
\end{eqnarray} 
 Hence for every $ a\in C^*(G_{5,3}) $ we have that
 \begin{eqnarray}\label{noesigG53}
\noop{\si_{\rh,\mu}(a)}&=&\sup_{\be\in\R}\noop{\pi_{\ell_{\be,\rh,\mu}}(a)} . 
\end{eqnarray}
\end{definition}
 \begin{definition}\label{vkdefG53}$  $\rm 
\begin{itemize}\label{defsiptwG53}
\item   Let $C_{\ol \O}= CB(L(\ol\O),\B(L^{2}(\R^{2}))) $ be the $ C^* $-algebra of all
continuous, uniformly bounded mappings $ \ph:L(\ol\O)\mapsto \B(L^{2}(\R^2)) $ from
the locally compact space $L(\ol\O)$ into the algebra of bounded linear
operators $ \B(L^2(\R^2)) $ on the Hilbert space $ L^2(\R^2) $. By the Theorem \ref{topodualG53} 
we observe  that for any $ a\in C^*(G_{5,3}) $, the operator field $ \hat
a\res{L(\ol\O)} $ is contained in $ C_{\ol\O} $.
 Furthermore, for   $ \ell=\rh C^*+\mu U^*\in \g_{5,3}^*$, 
we obtain a representation $ \tilde\si_{\rh,\mu}=\tilde\si_\ell $ on the Hilbert space $ L^2(\R^2) $ of the algebra $ C_{\ol \O} $ 
defined by: 
 \begin{eqnarray}\label{sielldefG53}
\tilde\si_\ell(\ph)\xi:=U_\ell\inv\left(\int_\R^\oplus u_{\be,\rh,\mu}^*\circ \ph(\be,\mu)\circ 
u_{\be,\rh,\mu}\big(U_\ell(\xi)(\be)\big)d\be\right), \ph\in C_{\ol \O}.
\end{eqnarray}
\item  Define for $ k\in \N$ and  $\ph\in C_{\ol\O}$ the linear
operator  $ \tilde \si_{k,\ol\O}(\ph)$ by
\begin{eqnarray}\label{deftildesikG53}
   \tilde\si_{k,\ol\O}(\ph):=\sum_{i\in \Z}\sum_{j\in \Z}M_{V^{k}_{i,j}}\circ
\tilde\si_{{(g^{k}_{i,j}})\inv\cdot p^{k}_{i,j}}(\ph)\circ M_{U^{k}_{i,j}}, 
\end{eqnarray}
where $ \tilde\si_\ell$ for $ \ell=(g^{k}_{i,j})\inv\cdot p^{k}_{i,j}, $ is an in Equation
(\ref{sielldefG53}). For $ a\in C^*(G_{5,3}) $ we have that 
$
 \si_{k,\ol\O}(a)=\tilde\si_{k,\ol\O}(\wh a\res {L(\ol O)}).
$ 
\end{itemize}
 \end{definition}
 \begin{proposition}\label{progencondG53}
 Let $a\in C^*(G_{5,3})$. Then: 
 \begin{eqnarray*}
  \underset{k\to\iy}{\lim}\noop{\pi_{\ell_k}(a)-\si_{k,\ol\O}(a)}=0.
 \end{eqnarray*}
\end{proposition}
\begin{proof}
 Let $\ve>0.$ Take first $F\in L_c^1(G_{5,3}).$ Let us choose a compact subset $K\subset\p^*$ and an $M>0$ such that the function 
 $\R^2\times\p^*\ni((x,y),p)\to\wh F^P(E(x,y),p)$ is supported in $[-M,M]^2\times K.$ By Proposition \ref{suppcompG53} 
 we have  for $k$ large enough:  
 \[\pi_{\ell_k}(F)\circ M_{U_{i,j}^k}= M_{V_{i,j}^k}\circ\pi_{\ell_k}(F)\circ M_{U_{i,j}^k},\ i,j\in\Z.\]
 The kernel function $F_k$ of the operator $\pi_{\ell_k}(F)\circ M_{U^{k}_{i,j}}-M_{V^{k}_{i,j}}\circ
\tilde\si_{{(g^{k}_{i,j}})\inv\cdot p^{k}_{i,j}}(F)\circ M_{U^{k}_{i,j}}$ 
 is given by 
 \begin{eqnarray*}
  F_k(s,t)  &=&1_{V_{i,j}^k}(s)1_{U_{i,j}^k}(t)\left(\wh F^P(st\inv,t\cdot p_k)-\wh F^P(st\inv,t(g_{i,j}^k)\inv\cdot p_{i,j}^k)\right)
 \end{eqnarray*}
Since the function $(s,p)\to \val{\wh F^P(s,p)}^2$ is in $C_c^\iy(G_{5,3}/P,\p^*)$ there exists a non-negative continuous function with 
compact support $\varphi:G_{5,3}/P\to\R_+$ such that for any $q,p\in\p*,\ s\in G_{5,3}/P:$
\[\val{\wh F^P(s,q)-\wh F^P(s,p)}\leq\va(s)\no{q-p}.\]
It follows then from Lemma \ref{lemma1G53} and Lemma \ref{lemma2G53} that for $k\in\N$ large enough, $i,j\in\Z,\ s\in G_{5,3}/P:$
\begin{eqnarray*}
\val{F_k(s,t)} &\leq&\left|\wh F^P(st\inv,t\cdot p_k)-\wh F^P(st\inv,t(g_{i,j}^k)\inv\cdot p_{i,j}^k)\right|\\
 &\leq&3\ve_k^{\frac 1 2}\va(st\inv).
\end{eqnarray*}
Using now Young's estimate, we see that for $k$ large enough and $i,j\in\Z:$
\begin{eqnarray*}
& &\noop{\pi_{\ell_k}(F)-\tilde\si_{k,\ol\O}(F)}
  \leq3\ve_k^{\frac 1 2}\no\va_1.
\end{eqnarray*}
\end{proof}
\subsubsection*{$\bullet$ Passing from $\GA_1^{5,3}$ to $\GA_0^{5,3}$.}
\begin{definition}\label{}
\rm Choose a Schwartz-function
$\eta$ in $\S(\R)$ with $L^2-$norm equal to $1$. For $(\al,\be)\in\R^2$ we
define the function $\eta_{\mu}(\al,\be)$ by
\begin{eqnarray}\label{etaG53}
  & &\eta_{\mu}(\al,\be)(s):= \val{\mu}^{\frac{1}{4}}
e^{2i\pi s\al}\eta( \val{\mu}^{\frac{1}{2}}(s+\frac{\be}{\mu})),\ s\in\R.
\end{eqnarray}

 \end{definition}
\begin{lemma}
 Let $\xi\in\S(\R)$. Then,
$$\xi=\frac{1}{\val{\mu}}\int_{\R^2}\langle\xi,\eta_\mu(\al,\be)\rangle\eta_\mu(\al,\be)d\al d\be.$$ 
\end{lemma}
\begin{definition}
 $1.$\ For all $(\al,\be)\in\R^2$ and $k\in\N$, let $P_{{\mu}(\al,\be)}$ be the
orthogonal projection onto  the one dimensional subspace
$\C\eta_{\mu}(\al,\be)$.\\
$2.$ Define for $h\in C_0(\R^2)$  the linear operator
\begin{eqnarray}\label{nuPG53}
 \sigma_{\rh,\mu}(h):=\frac{1}{ \val{\mu}}\int_{\R^2}
h(\al A^*+\be B^*+ \rh C^*)P_{\mu(\al,\be)}d\al d\be.
\end{eqnarray}
\end{definition}
\begin{proposition}(see Proposition 2.11 in \cite{Lud-Tur})
 \begin{enumerate}
  \item For every ${(\rh,\mu)}\in\GA^{5,3}_1$ and $h\in\S(\R^2)$ the integral $(\ref{nuPG53})$ converges in operator norm.
\item   $ \si_{\rh,\mu}(h) $ is compact and
$\no{\si_{\rh,\mu}(h)}_{op}\leq\no{h}_{\iy}$.
\item  The mapping $ \si_{\rh,\mu}: C_0(\R^{2})\to \F_1 $
is involutive, i.e. $ \si_{\rh,\mu}(h^*)=\si_{\rh,\mu}(h)^* , h\in C^*(\R^{2})$, where by
$\si_{\rh,\mu}$ we denote also  the extension of $\si_{\rh,\mu}$ to $C_0(\R^{2})$.
 \end{enumerate}

\end{proposition}

\begin{theorem}
 Let $a\in C^*(G_{5,3})$ and let $\va$ be the operator field
$\va=\F(a)$. Then the function $ \va(0):(\al,\be)\to \F(a)(\al,\be) $ is contained $ C_0(\R^2) $. 
Let $ \ol\ga=(\ga_k=(\rh_k,\mu_k))_k $ be a properly converging sequence in $ \GA^{5,3}_1 $ 
having its limit set $ L(\ol\ga)=\R A^* +\R B^* +\rh C^* $ in $ \GA^{5,3}_0 $.  
Then
\begin{eqnarray}
 \underset{k\to\infty}{\lim}\Vert
\va(\ga_k)-\sigma_{\ga_k}(\va(\rh))\Vert_{op}=0.
\end{eqnarray}
\end{theorem}
\begin{proof}
 The proof is the same as that of  Theorem $2.12$ in [Lud-Tur]. 
\end{proof}

\section{The $C^*-$algebra of the group $G_{5,4}.$}

   Recall that the  Lie algebra of $\g_{5,4}$ is spanned by the basis
$\B=\{A,B,C,U,V\}$ equipped with the Lie brackets
\begin{eqnarray}\label{defofbracg54}
 \nn [A,B]= C, [A,C]= U,  [B,C]= V.
\end{eqnarray}

This Lie algebra has  a two-dimensional centre $\z=\text{span}\{U,V\}.$ 
The group $G_{5,4}=\exp(\g_{5,4})$ can be realized as $\R^5$ with 
the Campbell-Baker-Hausdorff multiplication
\begin{eqnarray}\label{defofmulG54} 
\nn &&(a,b, c, u, v)\cdot(a',b', c' , u' , v')\\
\nn &=&(a+a',b+b',c+c'-a'b,u+u'-a'c+\frac{a'^2b}{2},v+v'+\frac{bc'} { 2 } -\frac{b'c}{2}+\frac{a'b'b}{2}).
\end{eqnarray}
For all $(a,b,c,u,v)\in G_{5,4},\ (\al,\be,\rh,\mu,\nu)\in\g_{5,4}^*$ we obtain the following expression for $Ad^*(a,b,c,u,v):$
\begin{eqnarray*}\label{defoforbG54}
&&\nn\Ad^*((a,b,c,u,v))(\al,\be,\rho,\mu,\nu)\\ 
&=&(\al- b\rho-
c\mu-\mu\frac{ab}{2}-\nu\frac{b^2}{2},
\be+\rh a-\nu c+\nu\frac{ab}{2}+\mu\frac{a^2}{2},
\rho+\mu a+\nu b,\mu,\nu).
\end{eqnarray*}
We give now a parameterization of the co-adjoint orbits:
 The generic elements $\ell=(\al,\be,\rh,\mu,\nu)$ in $\g_{5,4}^*$, are those for which a non-zero value $r_{\mu,\nu}=\sqrt{\mu^2+\nu^2}.$  
As in (\ref{basisp}) we take a new basis $\B_{\mu,\nu}$ of $\g_{5,4}.$ For that, letting $\tilde\mu:=\frac{\mu}{r_{\mu,\nu}},\ 
\tilde\nu:=\frac{\nu}{r_{\mu,\nu}}$ we put:
$$\B_\ell=\B_{\mu,\nu}:=\left\{A_{\mu,\nu}=A_\ell:=\tilde\mu A+\tilde\nu B,B_{\mu,\nu}=B_\ell:=-\tilde\nu A+\tilde\mu B,C,U,V\right\}.$$
In the dual basis $\B_\ell^*$ of the basis $\B_{\ell}$ the orbit $\O_\ell$ of the element $\ell=\ell_{\be,\mu,\nu}=\be B_\ell^*+\mu U^*+\nu V^*$
is given by: 
\begin{eqnarray}\label{genlayorG54}
 \O_{\ell}&=&\left\{a A_\ell^* +(\be+\frac{c^{2}}{2r_{\mu,\nu}}) B_\ell^*+c C^*+\mu U^*
+\nu V^*, a, c \in\R \right\}. 
\end{eqnarray}
It follows from this description that the function
$$Q:\ell=aA_\ell^*+bB_\ell^*+cC^*+\mu U^*+\nu V^*\to 2br_\ell-c^2$$ 
is $G_{5,4}-$invariant on this set. The stabilizer $\g_{5,4}(\ell)$ of $\ell$ is the sub-algebra 
$$\g_{5,4}(\ell)=\text{span}\{B_\ell,U,V\}.$$
We denote by $\GA_2^{5,4}$ the orbit space of this family of generic co-adjoint orbits parameterized by the set 
$$\GA_2^{5,4}:=\{\ell_{\be,\mu,\nu}\equiv(\be,\mu,\nu),\be\in\R,\ (\mu,\nu)\in\R^2\setminus\{(0,0)\}\}$$
Since $G_{5,4}/\exp(\z)=H_1$ we can decompose the orbit space $\g_{5,4}^*/G_{5,4},$ and hence also the dual space $\wh G_{5,4}$, 
into the disjoint union 
$$\g_{5,4}^*/G_{5,4}=\GA_2^{5,4}\dot\cup\GA_1^1\dot\cup\GA_0^1,$$
where $\GA_0^1$ and $\GA_1^1$ are as in (\ref{defofganHeis}).
\subsection{Limit sets of properly converging sequences in $\GA_2^{5,4}$.}
\begin{theorem}\label{goinginfG54}
  The sequence $ (\O_{\be_k, \mu_k,\nu_k})_k $ goes to infinity if
and only if the real sequence $ (\sqrt{\mu_k^2+\nu_k^2}+\ve_k \be_k+(\ve_k-1)\be_kr_k)_k $ goes to
infinity, where for $ k\in\N $:
\begin{equation}\label{}
\nonumber \ve_k:=\left\{\begin{array}{cc}
1 &\text{ if }\be_k>0\\
0 &\text{ if  }\be_k\leq 0.\\  
\end{array}
\right.
\end{equation}
\end{theorem}
\begin{proof}
 Suppose that the sequence of orbits does not tend to infinity. Then there is a convergent sub-sequence $(r_{k_j})$ and convergent sequence 
 $(c_j)$ such that $(\frac{2\be_{k_j}r_{k_j}+c_j^2}{2r_{k_j}})$ is convergent. Multiplying by $2r_{k_j}$ we see that $(2\be_{k_j}r_{k_j}+c_j^2)$ 
 is convergent and hence $(\be_{k_j}r_{k_j})$  and $(\be_{k_j})$ are convergent. Hence 
 $(\sqrt{\mu_{k_j}^2+\nu_{k_j}^2}+\ve_{k_j} \be_{k_j}+(\ve_{k_j}-1)\be_{k_j}r_{k_j})$ is convergent and so 
 $(\sqrt{\mu_k^2+\nu_k^2}+\ve_k \be_k+(\ve_k-1)\be_kr_k)$ does not tend to infinity.\\
 Conversely, suppose that $(\sqrt{\mu_k^2+\nu_k^2}+\ve_k \be_k+(\ve_k-1)\be_kr_k)$ does not tend to infinity. 
 Then there is a convergent sub-sequence $(r_{k_j})$ such that $(\be_{k_j}r_{k_j})$ is also convergent. We may choose convergent sequence 
 $(c_j)$ such that $2\be_{k_j}r_{k_j}+c_j^2=0$ for all $j$. Then $(0,0,c_j,\mu_{k_j},\nu_{k_j})\in\O_{k_j}$ and the sequence of functionals  
converges as $j\to\iy.$ Hence the sequence of orbits $(\O_k)$ does not tend to infinity.  
\end{proof}

\begin{theorem}\label{limitsetgenG54}
$  $
\begin{enumerate}
\item On the set $\GA_2^{5,4}$ the dual topology is Hausdorff.
 \item Let  $\ol\O=(\O_{\ell_k}=\O_{\be_k B_k^*+\mu_k U^*+\nu_k V^*})_k$ (where $ B_k:=B_{\ell_k}, k\in\N $) be a 
sequence in $ \GA_2^{5,4} $ with  $\lim_k r_k=\lim_k\sqrt{\mu_k^2+\nu_k^2}=0$.  We can assume (passing if necessary to a sub-sequence) 
that the real sequence $ (\tilde \mu_k)_k $  (res. $ (\tilde\nu_k)_k$) converges to $ \tilde\mu\ (\textrm{resp. to }\tilde\nu )$. 
Then the sequence of vectors $ (A_k=A_{\ell_k})_k $, resp. $ (B_k=B_{\ell_{k}})_k $ converges  to the vector 
$ A_\iy=\tilde\mu A+\tilde \nu B $ (resp. to $ B_\iy=-\tilde\nu A+\tilde\mu B $). 
\begin{itemize}\label{}
\item If $ \ol\O $ 
 has  a limit, then
$d:= \lim_{k\to\iy}(-2\be_kr_k=:d_k)$ exists and $ d\geq 0 $. 
\item  Suppose now that   $ \ol \O=(\O_{\ell_k}=\O_{\be_k B_k^*+\mu_k U^*+\nu_k V^*})_k$ is a properly converging sequence in 
$ \GA_2^{5,4} $.  
If $ \ol \O $ admits a limit in $ \GA_1^1\cup \GA_0^1 $, then its limit set  $ L(\ol\O) $ is given by:
 \begin{enumerate}\label{}
\item if $ d\ne 0 $ then $ L(\ol\O)=\left\{\O_{\sqrt  {d}}, \O_{-\sqrt {d}}\right\}$.
\item  if $ d=0 $, then the number  $\be_\iy :=\limsup_k  \be_k$ is contained in  $ [-\iy, +\iy[  $ and 
\begin{equation}\label{}
L(\ol\O)=\left\{ 
\begin{array}{cc}
\R A_\iy ^*+[\be_\iy,+\iy[ B_\iy^* &\text{ if }\be_\iy\in\R\\
\R A_\iy ^*+]-\iy,+\iy[ B_\iy^*=\GA_0^1&\text{ if }\be_\iy=-\iy. \\
\end{array}
\right.
\end{equation}
\end{enumerate}
\end{itemize}
\end{enumerate}

\end{theorem}
\begin{proof}$  $
\begin{enumerate}\label{}
 \item The point 1) is evident.
\item
\begin{itemize}
\item Let $ \ol\O $ be such a sequence in $ \GA_2^{5,4} $ having a limit.
Let $ \ell $ be a point in a  limit orbit of the sequence $ \ol\O $ and 
let $ m_k=a_k A_k^* +(\be_k+\frac{c_k^2}{2r_k})B_k^*+c_kC^*+\mu_k U+\nu_k V^*\in \O_k, k\in\N  $, such that $ \lim_k m_k=\ell $. 
Then:
\begin{eqnarray}\label{}
\nn \ell(B_\iy)&=&\lim_k m_k(B_k)\\
\nn &=&\lim_k \frac{2\be_k r_k+m_k(C)^{2}}{2r_k}.
\end{eqnarray}
This shows that $ \lim_k (2\be_k r_k+m_k(C)^{2})=0 $ since $ \lim_k r_k =\lim_k
\sqrt{\mu_k^{2}+\nu_k^{2}}=0$, i.e. $ \lim_k (-2\be_k r_k)=\ell(C)^2=:d $.
\item  
\begin{enumerate}\label{}
\item Let now $ \ol\O $ be properly convergent and suppose that $ d>0 $. 
For $ k $ large enough, we have that $ \be_k> 0 $ and then the element $
m_{k,\pm}:=\pm\sqrt {-2\be_kr_k}C^*+ \mu_k U^*+\nu_k V^*$
 of $ \O_k $ converges to $ \pm\sqrt {d} C^* $. Hence the orbits $ \O_{\pm
\sqrt {d}} $ are contained in $ L(\ol \O) $. 
On the other hand, every other $ \ell $ in the limit set $ L(\ol\O) $ satisfies the
relation $d=\ell(C)^{2}$,
 which means that $ \ell(C)=\pm\sqrt{ d} $. Hence $
L(\ol\O)=\{\O_{\sqrt{d}},  \O_{-\sqrt {d}}\} $. 
\item
If now $ d=0 $, we see in a similar manner that the limit of the sequence
$ \ol \O $ must be characters, i. e. vanish on $ C $. 
Choose a sub-sequence (also denoted by $ \ol\O $ for simplicity), such that $
\lim_k \be_k=\be_\iy$. 
Take  any sequence 
\[ \left(m_k=a_k A_k^*+(\frac{2\be_k r_k+c_k^2}{2r_k})B_k^*+ c_k C^*+\mu_k
U^*+\nu_k V^*\in\O_k\right)_k \]
 which converges to  $ \ell=aA^*+b B_\iy^*\in L(\ol\O) $ for some $ a,b\in\R $. 
Since
\begin{eqnarray*}
b=\lim_{k\to\iy}(\be_k+\frac{c_k^2}{2r_k})=\lim_k \be_k+\lim_k \frac{c_k^2}{2r_k}\geq \lim_k \be_k=\be_\iy,
\end{eqnarray*}
it follows that $\be_\iy\leq b<+\iy$.  On the other hand, for any $
a\in\R,\  b> \be_\iy $,  we have that $ b\geq \be_k $ for $ k $ large enough and the sequence 
\[ \left(m_k=a A_k^*+bB_k^*+ \sqrt{(b-\be_k)2r_k} C^*+\mu_k
U^*+\nu_k V^*\in\O_k\right)_k \]
converges to 
$ aA_\iy^*+b B_\iy ^* $, since $ \lim_k \be_k r_k=0 $. Hence, since $L(\ol\O)$ is closed, 
\[ L(\ol \O)=\R A_\iy ^*+[\be_\iy,+\iy[ B_\iy^* (\textrm{ resp. }L(\ol \O)=\R A_\iy ^*+]-\iy,+\iy[ B_\iy^* ).    \]
 \end{enumerate}
 \end{itemize}
\end{enumerate}
\end{proof}
\subsection{The Fourier transform.}
Let
$\ell\in \GA_2^{5,4}$, then $\dim {\O_\ell} =2$. A polarization at $ \ell=\ell_{\be,\mu,\nu}= \be B_\ell^*+\mu U^*+\nu V^*$
is given by $ \p_\ell=\p_{\mu,\nu}=\textrm{span}\{B_\ell,C,U,V\}$.
Then $ G_{5,4}=\exp({\R A_\ell})P_\ell $ and $ P_\ell=\exp({\R B_{\mu,\nu}})P $ as
topological products.
We take $ \pi_\ell:=\ind{P_\ell} {G_{5,4}} \ch_\ell .$ Its Hilbert space is
isomorphic to $ \l2\R $ and for 
$ g= \exp({s A_\ell})p, s,u \in \R,\ p\in P_\ell$  and  $ \xi\in \l2\R $ we have
\begin{eqnarray}\label{}
\nn \pi_\ell(g)\xi(u)&=&e^{-2\pi i \langle{\textrm{exp}{(u-s)A_\ell})\cdot
\ell},{\log(p_\ell)}\rangle}\xi(u-s),
\end{eqnarray}
and so for $ F\in \l1 {G_{5,4}} $:
\begin{eqnarray}\label{opkergenG54}
 \nn\pi_{\ell}(F)\xi(t)&=& \int_{G_{5,4}}F(g)\pi_{\ell}(g)\xi(t)dg\\
 &=& \int_{\R}\widehat F^{P_\ell}(s-t,s.\ell\res{\p_\ell})\xi(s)ds.
\end{eqnarray}
Here 
\begin{eqnarray*}
\wh F^{P_\ell}(s,q)=\int_{P_\ell}F(\exp{(s A_\ell)}p)\ch_q(p)dp, q\in  \p_\ell^*, s\in\R.
\end{eqnarray*}
and $ s\cdot q=\exp(s A_{\mu,\nu})\cdot q,\  q\in \p_\ell^* $.
For $ F\in L_c^1(G_{5,4}) $, the function $ \wh F^{P_\ell} $ is of compact support in $ s\in\R $ and 
Schwartz in the variable $ q\in  \p_\ell^* $.
\begin{definition}
  The Fourier transform $\hat a=\F(a)$ of an element
$a\in C^*(G_{5,4})$ is defined as the field of bounded linear operators  over the dual space of $ G_{5,4}$, but
 where we put  all the unitary characters of $ G_{5,4}$ together to form 
 the representation $ \pi_0 $ $($here $\pi_0$ is the left regular representation of the group $G_{5,4})$.   
 This gives us the set $ \GA_2^{5,4}\cup\GA_1^1\cup\GA_0^1 $ and we define for
 $ a \in C^*(G_{5,4}) $ the operator field:
\begin{eqnarray}\label{deftranfoG54}
\nn &&\wh
a(\be,\mu,\nu)=\F(a)(\be,\mu,\nu):=\pi_{\be,\mu,\nu}(a)
\in\K(\l2\R),\ (\be,\mu,\nu)\in \GA_2^{5,4};\\
\nn &&\wh a(\rh)=\F(a)(\rh):=\pi_{\rh}(a)
\in\K(\l2\R),\ \rh\in \GA_1^1;\\
\nn&&\wh a(0)=\F(a)(0):=\pi_0(a)\in C^*(\R^2)\subset \B(\l2{\R^2}).
\end{eqnarray}
\end{definition}
\subsection{The continuity condition.}
We have seen in Theorem\ref{limitsetgenG54}, that the topology of the sub-set $\GA_2^{5,4}$ is Hausdorff. This means for the F's in $C^*(G_{5,4}),$ 
that the functions $\pi\to\noop{\pi(F)}$ are continuous on this set.
\begin{theorem}
 The mapping $\GA_2^{5,4}\to\B(L^2(\R)):\ell\mapsto\pi_\ell(F)$ is norm-continuous for all $F\in C^*(G_{5,4}).$
\end{theorem}

\subsection{Passing from $\GA_2^{5,4}$ to $\GA_1^1\cup\GA_0^1.$}
\begin{definition}\label{dnenull}
\rm   Let $ \ol\O=(\O_{\ell_k=\be_k B_{k}^*+\mu_k U^*+\nu_k V^*}) $ be a
properly converging sequence in $ \GA_2^{5,4} $ such that $ \lim_k (r_k=\sqrt {\mu_k^2+\nu_k^2})=0 $. We can assume 
(passing if necessary to a subsequence) that   $ \lim_k \tilde u_k=\tilde\mu $ and $\lim_k \tilde\nu_k =:\tilde{\nu}$ exist too. 
Let 
\begin{eqnarray*}
d_k:=-2\be_k r_k>0 \textrm{ and } t_k=\sqrt{\frac{-2\be_k}{r_k}}=\frac{\sqrt {d_k}}{r_k}, k\in\N.
\end{eqnarray*}

\noindent
By Theorem \ref{limitsetgenG54} and its notations,  $ d=\lim_k d_k  $ exists. If $ d=0 $, then  $ \be_\iy:=\lim_k \be_k  $ 
exists in $ [-\iy, +\iy[ $ and the  limit set of $ \ol\O $ is 
the set $ L(\ol\O)=\{\R A_{\mu,\nu}^*+[\be_\iy,+\iy] B_{\mu,\nu}^*\}  $ where $ A_\ell^*=\tilde \mu A^* +
\tilde\nu B^* $ and $ B_\ell^*=-\tilde \nu A^* +\tilde\mu B^* $. 
Otherwise, i.e. if $ d\ne 0,$  the limit set $ L(\ol\O) $ is given by $ L(\ol\O)=\{{\sqrt d C^*}, {-\sqrt d C^*} \}$.
  Recall that:
\begin{eqnarray*}
\exp{\left((t_k+s) A_k\right)}\cdot p_k&=&\left(\be_k
+\frac{r_k(s+t_{k})^{2}}{2}\right)B_k^*+\left((s+t_{k})r_k\right)C^*+\mu_k U^*+\nu_k
V^*\\
\nn  &=&s\left(
r_k\frac{
s}{2}+r_kt_{k}\right)B_k^*+((s+t_{k})r_k)C^*+\mu_k
U^*+\nu_k
V^*, k\in\N.
\end{eqnarray*}
\end{definition}
\subsubsection*{\bf{If $d\ne 0.$}}
We consider first the case where $ d:=\lim_k d_k \ne 0$. 
Let 
\begin{eqnarray*}
q_k=\sqrt {d_k}C^*, k\in\N^*.
\end{eqnarray*}
Let us compute:
\begin{eqnarray}\label{compdiff}
\nn&&\exp{((s+\pm t_k)A_k)}\cdot p_k-\exp{(s A_k)}\cdot (\pm q_k )\\
\nn&=&\left(\be_k+\frac{r_k (s+\pm t_k)^{2}}{2}-\pm\sqrt {d_k}s\right)B_k^*+\left(
r_k(\pm t_k+s)-\pm\sqrt {d_k}\right)C^*+\mu_k U^*+\nu_k V^*\\
\nn &=&\left(\frac{r_k s^{2}}{2}\right)B_k^*+(r_k s)C^*+\mu_k
U^*+\nu_k V^*.
\end{eqnarray}
Let $ (R_k)_k $ be a sequence in $ \R_+ $ such that $ \lim_k
R_k^2 r_k=0, \lim_k R_k=+\iy $. Let  for $ k\in\N: $
\begin{eqnarray*}
J_k&=&[-{R_k},{R_k}],\\ 
I_{k,\pm}&:=&(\pm t_k+J_k),\\
I_{k,2,\pm}&:=&(\pm t_k+2J_k).
\end{eqnarray*}
Our condition on $ R_k $ tells us that $ \lim_k \frac{ R_k}{t_k}=0 $ and so $ I_{k,+}\cap I_{k,-}=\es $ for $ k $ large enough.
\begin{lemma}\label{getsbig}
Let $ K $ be a compact subset of $ \p^* $. Then 
$ \exp{( (\R\setminus (I_{k,+}\cup I_{k,-}))A_k)}\cdot{\ell_k}\res\p\cap K =\es$ for $ k $ large enough.
 \end{lemma}
\begin{proof} Take $ R>0 $ such that $K \subset [-R,R] $.
 For $ s>0, s\not \in  I_{k,+}$ we have either $ s>t_k+R_k $ or $ 0\leq s\leq t_k-R_k $ (for $ k $ large enough): 
 In the first case:
\begin{eqnarray*}
\vert\exp{(s A_k)}\cdot \ell_k(B_k)\vert&=&\vert((s-t_k)\cdot(t_k\cdot \ell_k))(B_k)\vert\\
&=&
\left|(s-t_k)\left(
r_k\frac{(
s-t_k)}{2}+r_kt_{k}\right)\right|\\
\nn &\geq& R_k \sqrt{d_k}\geq \frac{R_k\sqrt d}{2} \ (\textrm{ for }k\textrm{ large enough}). 
\end{eqnarray*}
In the second case:
\begin{eqnarray*}
\vert\exp{(s A_k)}\cdot \ell_k(B_k)\vert&=&\vert((s-t_k)\cdot(t_k\cdot \ell_k))(B_k)\vert\\
&=&
\left|(s-t_k)\left(
r_k\frac{(
s-t_k)}{2}+r_kt_{k}\right)\right|\\
\nn &\geq& R_k \left(t_kr_k-\frac{(t_k-s_k)r_k}{2} \right)\geq \frac{R_k\sqrt d}{4}\\
&&\ (\textrm{ for }k\textrm{ large enough}).
\end{eqnarray*}
Similarly for $ s<0,\ s\not\in I_{k,-} $.
This means that for $ k $ large enough, $\exp({t A_k})\cdot {\ell_k}\not\in
 K  $ for $ t\not\in I_{k,\pm} $.
 \end{proof}
\begin{definition}\label{defmksik}
\rm  Let $C_{\ol \O}= CB(L(\ol\O),\B(L^{2}(\R))) $ be the $ C^* $-algebra of all
continuous, uniformly bounded mappings $ \ph:L(\ol\O)\mapsto \B(L^{2}(\R)) $ from
the locally compact space $L(\ol\O)$ into the algebra of bounded linear
operators $ \B(L^2(\R)) $ on the Hilbert space $ L^2(\R) $.

Let for $ k\in \N^*, $ 
\begin{eqnarray*}
\si_{k,\pm}=\ind{P_k}{G_{5,4}}\ch_{\exp(\pm t_k A_k)\cdot (\sqrt {d_k}C^*)},\  \pi_{\sqrt {d_k }C^*}=\ind{P_k}{G_{5,4}}\ch_{\sqrt {d_k}C^*}
\end{eqnarray*}
and let $ u_{k,\pm} $ be the unitary intertwining operator between $ \pi_{\sqrt {d_k }C^*} $ and $ \si_{k,\pm} $. 
\begin{enumerate}\label{}
\item  Let for $ \ph\in C_{\ol\O},\ k\in\N, $
\begin{eqnarray}\label{deftilsi1G54}
\nn\tilde\si_{k,\ol \O}(\ph)=\tilde  \si_{k}(\ph)&:=&M_{I_{k,2,+}}\circ u_{k,+}\circ \ph(\sqrt {d_k})\circ u_{k,+}^*\circ M_{I_{k,+}}\\
&+&M_{I_{k,2,-}}\circ u_{k,-}\circ \ph(-\sqrt {d_k})\circ u_{k,-}^*\circ M_{I_{k,-} }\in
\B(L^{2}(\R)).
\end{eqnarray}
\item  For $ a\in C^*(G_{5,4}) $ let 
\begin{eqnarray*}
 \si_{k,\ol\O}(a)&:=&\tilde\si_{k,\ol \O}(\wh a\res{\GA_1^1}).
\end{eqnarray*}
 \end{enumerate}
 \end{definition}
\begin{theorem}\label{condigenG54}
Let $ \ol\O=(\O_{\ell_k=\be_k B_{k}^*+\mu_k U^*+\nu_k V^*}) $ be a  properly
converging sequence in $ \GA_2^{5,4} $ with limit set $ L=\{\O_{\sqrt d}, \O_{-\sqrt
d}\} $ where $ d=-2\lim \be_k\sqrt{\mu_k^{2}+\nu_k^{2}}=\lim_k (-2\be_k r_k),
k\in\N $. 
Then for every $ a\in C^*(G_{5,4}) $, we have that 
\begin{eqnarray*}
\lim_{k\to\iy}\noop{\pi_{\ell_k}(a)-\si_{k,\ol\O}(a)}=0.
\end{eqnarray*}
 \end{theorem}
\begin{proof}
Let $ F\in L_c^1(G_{5,4}) $. Let for $ k\in\N^* $ \begin{eqnarray*}
\p_k:=\p_{\ell_k},\ P_k=P_{\ell_k}=\exp(\R B_k)\cdot P=\exp{(\p_k)}.
\end{eqnarray*}
The normal subgroup $ P_k=\exp({\p_k}) $ is  a polarization at $ \ell_k $ for every $ k $. Let for $ k\in\N^*: $
\begin{eqnarray*}
F_k(u,t):= \wh F^{P_k}(\exp(u A_k),\exp(t A_k)\cdot\ell_k{\res{\p_k}} ), u,t\in\R.
\end{eqnarray*}
Then the  kernel function $ K_k $ of the linear operator $ \pi_{\ell_k}(F) $ is given by: 
\begin{eqnarray*}
K_k(s,t)=F_k(s-t,t), s,t\in\R.
\end{eqnarray*}
 Since $ F\in  L_c^1(G_{5,4}) $, there is an $ M>0 $ such that $ F_k(s-t,t )=0 $, if $ \val {s-t}>M $ and together with Lemma \ref{getsbig}
we have therefore for $ k $ large enough that
\begin{eqnarray*}
\pi_{\ell_k}(F)&=&M_{I_{k,2,+}}\circ \pi_{\ell_k}(F)\circ M_{I_{k,+}}
+M_{I_{k,2,-}}\circ \pi_{\ell_k}(F)\circ M_{I_{k,-}}.
\end{eqnarray*} 
The kernel function $ F_k $ of the operator $ \pi_{\ell_k}(F)-\tilde\si_{k,\ol\O}(F) $ is given by 
\begin{eqnarray*}
F_k(s,t)&=&1_{I_{k,2,+} }(s)1_{I_{k,+}}(t)\Big (
\wh F^{P_k}(\exp((s-t)A_k),\exp ((t+t_k)A_k)\cdot \ell_k{\res{\p_k}})\\
&-&\wh F^{P_{k}}(\exp((s-t)A_k),\exp(tA_k)\cdot q_k)\Big)\\
&+&1_{I_{k,2,-} }(s)1_{I_{k,-}}(t)\Big(\wh
F^{P_{k}}(\exp((s-t)A_k),\exp((t-t_k)A_k)\cdot \ell_k{\res{\p_k}})\\
&-&\wh F^{P_{k}}(\exp((s-t)A_k),\exp(tA_k)\cdot (-q_k))\Big)
\end{eqnarray*}
Since $ F\in L_c^1(G_{5,4}) $, there exists a continuous function $ \va\geq 0 $ on $G_{5,4}/P_k $
with compact support, such that 
\begin{eqnarray*}
\vert \wh F^{P_k}(\exp(s A_k) ,\ell\res{\p_k})-\wh F^{P_k}(\exp(s A_k) ,\ell'\res{\p_k})\vert
\leq 
\va(s)\no{\ell\res{\p_k}-\ell'\res{\p_k}}, s\in G_{5,4}, \ell,\ell'\in\g_{5,4}^*.
\end{eqnarray*}
It follows for $ t=v+t_k\in I_{k,+}, s=u+t_k\in I_{k,2,+} $ that 
\begin{eqnarray*}
&& 
\vert \wh F^{P_{k}}(\exp{(u-v)}A_k),(v+t_k)\cdot \ell{_k{\res {\p_k}}})-\wh F^{P_{k}}(\exp{(u-v)}A_k),v\cdot
q_k)\vert 
\\
&&\leq
\va(u-v){\no{\exp((v+t_k)A_k)\cdot p_k-\exp(vA_k)\cdot q_k}}\\
&\leq&\left(\vert\frac{r_k v^{2}}{2}\vert +\vert  r_k v\vert+\vert\mu_k\vert
+\vert\nu_k\vert\right)\va(u-v)\\
&\leq&\left(\frac{r_k R_k^{2}}{2} +  r_k R_k+\vert\mu_k\vert
+\vert\nu_k\vert\right)\va(u-v), 
\end{eqnarray*}
for $ k $ large enough. Similarly for  $ t\in I_{k,-}, s\in I_{k,2,-} $. 
Since $ \lim_k r_k R_k^2 $=0, 
 Young's inequality implies  that  
\begin{eqnarray*}
\lim_{k\to\iy}\noop{\pi_{\ell_k}(F)-\tilde\si_{k,\ol \O}(F)}&=&0. 
\end{eqnarray*}
$  L_c^1(G_{5,4}) $ being  dense in $ C^*(G_{5,4}) $ the theorem follows.
\end{proof}
\subsubsection*{If $ d=0 $.}
We suppose now that $ \lim_k (d_k=-2\be_k r_k)=0 $. 
\begin{definition}\label{dnenull}
\rm   Let $ \ol\O=(\O_{\ell_k=\be_k B_{k}^*+\mu_k U^*+\nu_k V^*}) $ be a
properly converging sequence in $ \GA_2^{5,4} $. We can suppose that $ \lim_k \be_k=\be_\iy $
exists 
in $ [-\iy,+\iy[ $ and that $ \lim_k \tilde \mu_k=\tilde\mu $, $ \lim_k
\tilde\nu_k=\tilde\nu $. Let as before $ A_k:=\tilde\mu_k A+\tilde\nu_k B $,  $
B_k:=-\tilde \nu_k A+\tilde\mu_k B $ and $ A_\ell=\tilde\mu A+\tilde\nu B $, $
B_\ell=-\tilde \nu A+\tilde\mu B $. The limit set is
given by $L(\ol\O)=\R A_\ell^*+ [\be_\iy,\iy[ B_\ell^*$ if $ \be_\iy\in\R $ otherwise $ L(\ol O)=\GA_0^1 $.
\begin{enumerate}\label{}
\item 
Let  $ (\ve_k)_k\subset \R_+ $ be a decreasing sequence converging to $0$ such
that
$ \lim_k\frac{\ve_k}{r_k}=+\iy $, 

\item  let
\begin{eqnarray*}
t^{k}_j:=j\sqrt{\frac{2 \ve_k}{r_k}}, k\in\N,j\in\Z.
\end{eqnarray*}
 \end{enumerate}

Then:
\begin{eqnarray*}
&&\exp{(( t^{k}_j+s) A_k)}\cdot p_k\\&=&\left(\be_k
+r_k\frac{(t^{k}_j+s)^{2}}{2}\right)B_k^*+(r_k(t^{k}_j+s))C^*+\mu_k U^*+\nu_k V^*\\
\nn &=&\left(\be_k +j^2\ve_k+r_kj\sqrt{\frac{2
\ve_k}{r_k}}s+r_k\frac{s^{2}}{2}\right)B_k^*+\left(r_kj\sqrt{\frac{2
\ve_k}{r_k}}-r_ks\right)C^*+\mu_k U^*+\nu_k V^*\\
\nn &=&\left(\be_k +j^2\ve_k+j\sqrt{{2
\ve_k}{r_k}}s+r_k\frac{s^{2}}{2}\right)B_k^*+\left(j\sqrt{{2
\ve_k}{r_k}}+r_ks\right)C^*+\mu_k
U^*+\nu_k V^*.
\end{eqnarray*}
Let
\begin{eqnarray}\label{tjkpkdef}
\nn p^k_j&:=&\exp{(t^{k}_jA_k)}\cdot p_k=(\be_k +j^2\ve_k)B_k^*+(j\sqrt{{2
\ve_k}{r_k}})C^*+\mu_k U^*+\nu_k V^*,\\
\\
\nn q^{k}_j&:=&(\be_k +j^2\ve_k)B_k^*+j\sqrt{{2 \ve_k}{r_k}}C^*.
\end{eqnarray}
Let  $ s\in [t_{j}^k,t^k_{j+1}[, j\in \Z $. Then for $ k $ large enough:
\begin{eqnarray}\label{diffest2}
\nn\no{\exp{(( t^{k}_j+s) A_k)}\cdot p_k-\exp{(sA_k)}\cdot
q^k_j}&=&\val{r_k\frac{s^{2}}{2}}+\val{r_ks}+\val{\mu_k}+\val{\nu_k}\\
\nn
&\leq&\val{r_k\frac{(t^k_{j+1}-t^k_{j})^{2}}{2}}+\val{r_k(t^k_{j+1}-t^k_{j})}
+\val{\mu_k}+\val{\nu_k}\\
\nn &\leq& \ve_k+\sqrt{2r_k\ve_k}+\val{\mu_k}+\val{\nu_k}\\
&<&\ve_k^{\frac1 2}.
\end{eqnarray}
 \end{definition}
\begin{definition}\label{rkjdef}
Let for $ k\in\N$ and $j\in\Z $:
\rm   \begin{eqnarray*}
I_{k,j}&:=&[t_j^k,t_{j+1}^k[,\\
\nn I_k&:=&\bigcup_{j\in J_k}I_{k,j}.
\end{eqnarray*}
 \end{definition}
\begin{lemma}\label{getsbig2}$  $
 For all  $ R>0,j\in \Z $, we have that $R +I_{k,j}\subset
I_{k,j}\cup I_{k,j+1} $, for $ k $ large enough.
 \end{lemma} 
\begin{proof} 
This  follows from the fact that $ \lim_k(t^{k}_{j+1}- t^{k}_j)=\sqrt{\frac{2\ve_k}{r_k}}=\iy. $
\end{proof}
\begin{definition}\label{defmksik}
\rm  
\begin{enumerate}\label{}
$  $

\item  
Let for $ k\in \N $ and $j\in\Z$ 
\begin{eqnarray*}
\si_{k,j}:=\ind{P}{G_{5,4}}\ch_{\exp(\pm t_k A_k)\cdot q^k_j}
\end{eqnarray*}
and let $ u_{k,j} $ be the unitary intertwining operator between $ \pi_{-j\sqrt {2\ve_k r_k}C^*} $ and $ \si_{k,j} $. 

\item Let as before $ C_{\ol \O} $ be the $ C^* $-algebra of all continuous bounded  
mappings from $ L(\ol \O) $ into $\B(L^2(\R))  $.  Let for $ \ph\in C_{\ol \O} $:
\begin{eqnarray}\label{deftilsi2G54}
\tilde\si_{k,\ol \O}(\ph)=\sum_{j\in \Z} (M_{I_{k,j+1}}+M_{I_{k,j}})\circ
u_{k,j}\circ \si_{k,j}(\ph)\circ u_{k,j}^*\circ M_{I_{k,j}}.
\end{eqnarray}
\item  For $ a\in C^*(G_{5,4}) $ let 
\begin{eqnarray*}
\si_{k,\ol\O}(a)&=&\tilde\si_{k,\ol \O}(\wh a\res{L(\ol \O)}).
\end{eqnarray*}
 \end{enumerate}
 \end{definition}
 The proof of the next proposition is similar to that of Proposition \ref{sumsofop}.
 \begin{proposition}\label{boundedsi3}
The linear mappings $ \tilde\si_{k,\ol\O}, k\in\N, $ are bounded by $ 2 $.
 \end{proposition}
 \begin{theorem}\label{condigen2G54}
Let $ \ol\O=(\O_{\ell_k=\be_k B_{k}^*+\mu_k U^*+\nu_k V^*})_k $ be  a
properly converging sequence in $ \GA_2^{5,4} $ with the properties of Definition
\ref{dnenull}.
Then for every $ a\in C^*(G_{5,4}) $, we have that 
\begin{eqnarray*}
\lim_{k\to\iy}\noop{\pi_{\ell_k}(a)-\si_{k,\ol\O}(a)}=0.
\end{eqnarray*}
 \end{theorem}
\begin{proof}
Let $ F\in L_c^1(G_{5,4}) $. Then, 
for $ k $ large enough, we have by Lemma \ref{getsbig2} that
\begin{eqnarray*}
\pi_{\ell_k}(F)=\sum_{j\in \Z} (M_{I_{k,j}}+M_{I_{k,j+1}})\circ
\pi_{k}(F)\circ M_{I_{k,j}}.
\end{eqnarray*}
Therefore the kernel function $ F_{k,j} $ of the operator $(M_{I_{k,j}}+M_{I_{k,j+1}})\circ(
\pi_{k}(F)-\tilde\si_{k,j}(F))\circ M_{I_{k,j}}  $ is given by:
\begin{eqnarray*}
 F_{k,j}(s,t)&=& 1_{ I_{k,j}\cup I_{k,j+1}}(u+t^k_j) 
(\wh F^{P_{k}}(u-v,(v+t_j^k)\cdot p_k)-\wh F^{P_{k}}(s-t,v\cdot
q_j^k)),\\
&& \text{ with } s=u+t^k_j, t=v+t^k_j.
\end{eqnarray*}
Since $ F\in L_c^1(G_{5,4}) $ there exists a continuous function $ \va\geq 0 $ on $ \R $
with compact support, such that 
\[\vert \wh F^{P_{k}}(s-t,p)-\wh F^{P_{k}}(s-t,q)\vert
 \leq \va(s-t)\no{p-q}  \]
for every $ k\in\N $ and every $ p,q\in \p_k^* $.
Hence, by (\ref{diffest2})
\begin{eqnarray*}
\val{F_{k,j}(s,t)}&\leq & 1_{ I_{k,j}\cup I_{k,j+1}}(u+t^k_j)1_{I_{k,j}}(v+t^k_j) \\  
&&\ti \vert\wh F^{P_{k}}(u-v,(v+t_j^k)\cdot p_{k})-\wh F^{P_{k}}(s-t,v\cdot
q_j^k)\vert\\
\nn &\leq&\va(u-v)\no{(v+t_j^k)\cdot p_k-v\cdot {q_j^k}}\\
\nn
&\leq&\ve_k^{\frac1 2}
\va(u-v).
\end{eqnarray*}
It follows now from Young's inequality and from the properties of the sequence $
(I_{k,j})_j $ that
\begin{eqnarray}\label{goeszero5}
\noop{\pi_{\ell_k}(F)-\tilde\si_{k,\ol\O}(F)}\leq 2\ve_k^{\frac1 2}\no\va_1.
\end{eqnarray}
Since $  L_c^1(G_{5,4}) $ is dense in $ C^*(G_{5,4}) $ and since the mappings $ \si_{k,\ol\O} $
are all bounded in $ k $ by a fixed constant, it follows that
 relation (\ref{goeszero5}) also holds for $ a\in  C^*(G_{5,4})$.
\end{proof}

\section{ The $ C^* $-algebra of the group $G_{5,6}$.}
Recall   that the Lie algebra 
$\g_{5,6}$ is  spanned by the basis $\B=\{A,B,C,U,V\}$ equipped 
with the Lie brackets $$[A,B]= C, [A,C]= U,  [A,U]= V, [B,C]=V.$$
It has  a one-dimensional centre $\z=\R V$.
 The group $G_{5,6}=\exp(\g_{5,6})$ can be realized as $ \R^{5} $ with the
multiplication 
\begin{eqnarray}\label{cdotg56} 
\nn& & (a,b,c,u,v)\cdot(a',b',c',u',v')\\
\nn&=&(a+a',b+b',c+c'-a'b,u+u'-a'c+\frac{a'^2b}{2},v+v'-a'u+\frac{bc'}{2}-\frac{b'c}{2}+\frac{a'bb'}{2}+\frac{a'^2c}{2}-\frac{a'^3b}{6}).
\end{eqnarray}
We use the euclidean scalar product on $ \g_{5,6} $  to
identify $ \g_{5,6}^* $ with $ \g_{5,6}=\R^5 $ and we we obtain the
following expression for $ \Ad^*(a,b,c,u,v) $:
\begin{eqnarray*}\label{Ad*g56}
\nn& &\Ad^*((a,b,c,u,v))(\al,\be,\rho,\mu,\nu)\\
\nn&=&(\al- \rho b-
\mu c-\mu\frac{ab}{2}-\nu u-\nu\frac{
b^2}{2}-\nu\frac{ac}{2}-\nu\frac{a^2b}{6},
\be+\rho a+\mu\frac{a^2}{2}-\nu c+\nu\frac{ab}{2}+\nu\frac{a^3}{6},\\ 
\nn& &\rho+\mu a+\nu b+\nu\frac{a^2}{2},\mu+\nu a,\nu).
\end{eqnarray*}
We give now a description of the co-adjoint orbits:
\begin{enumerate}
 \item[] The generic orbits: if $\nu\ne0$. The orbit $ \O_{\nu} $  of the
element $\ell_{\nu}=(0,0,0,0,\nu) $ is given by:
\begin{eqnarray}\label{genorbitg56}
\nn \O_{\nu}&=&\{(a,b,c,u,\nu),\  a,b,c,u \in\R \}. 
\end{eqnarray}
 The stabilizer of $\ell_{\nu}$ is the set 
$\g_{5,6}(\ell_{\nu})=span\{
V\},$ 
we denote by $ \GA_3^{5,6} $ the orbit space of this layer and we parametrize it by
\begin{eqnarray}\label{defGA3G56}
\GA_3^{5,6}:=\{\ell_{\nu}\equiv\nu, \nu\in\R^*\}.
\end{eqnarray}
\end{enumerate}
Since $G_{5,6}/V=F_4$ we can decompose the orbit $\g_{5,6}^*/G_{5,6}$ and hence also the dual space $\wh G_{5,6},$ into the disjoint union
$$\g_{5,6}^*/G_{5,6}=\GA_3^{5,6}\dot\cup\GA_2^{4}\dot\cup\GA_1^4\dot\cup\GA_0^4.$$
\begin{theorem}\label{topodualG56}
 Let $\ol\O=(\O_{\nu_k})_k\subset\GA_3^{5,6}$ be a sequence, such that $\underset{k\to\iy}{\lim}\nu_k=0.$ Then $\ol\O$ 
is properly converging and  $L(\ol\O)=\GA_2^{4}\cup\GA_1^4\cup\GA_0^{4}.$
\end{theorem} 
\subsection{The Fourier transform.}
\begin{definition}
 Let $F\in L^1(G_{5,6}),$ the operator $\pi_{\ell}$ is a kernel operator with kernel function 
$$\widehat F^{P_{\ell}}(s,t,\ell|_{P_\ell})=\int_{P_\ell}F(spt\inv)\chi_\ell(p)dp,\ s,t\in G_{5,6}/P_{\ell},$$
where $P_\ell=\exp(\p_\ell)$ and $\p_\ell$ is a polarization at $\ell.$
\end{definition}
 For $\ell=(0,0,0,0,\nu)\in\GA_3^{5,6},$ therefore the abelian sub-algebra $\p=\text{span}\{C,U,V\}$ is a polarization at $\ell.$ 
We realize then $ \pi_{\ell,P}=\pi_{\ell} $ as $ \pi_{\ell}:=\ind P {G_{5,6}}\ch_\ell$. 
The Hilbert space $ \l2{G_{5,6}/P,\ell} $ is in fact isomorphic to $ \l2{\R^{2}} $:
let $ E:\R^{2}\to G_{5,6}, E(a,b):=\exp{(a A)}\exp{(b B)} $ and 
$S=\exp(\R  A)\exp(\R B)=E(\R\ti \R)$.  Then 
$G=S.P$ as topological product and the mapping $ U:\l2{G_{5,6}/P,\ell}\to \l2{\R^{2}}
$  defined by $ U\xi(t):=\xi(E(t)),t\in\R^{2}, $ is unitary. We identify now $
\pi_{\ell} $ 
with the corresponding representation on $ \l2{\R^{2}} $.
Let us compute the  operator $\pi_{\ell}(F)$ for $F\in
C^*(G_{5,6})$ explicitly.
For  $\xi\in L^2(\R^2)$, $t=(a',b')\in S,p\in P$ we have:
\begin{eqnarray}
\nn& &\pi_\ell(F)\xi(t)\\
\nn&=&\int_{G_{5,6}/P}\xi(s)\left(\int_PF(t ps\inv)e^{-2i\pi\langle s.\ell,p\rangle}dp\right)ds\ 
(\text{where }s.\ell=\Ad^*(s).\ell)\\
\nn&=&\int_{G_{5,6}/P}\widehat F^P(ts\inv, s.\ell|_\p)\xi(s)ds\\
\nn&=&\int_{G_{5,6}/P}\widehat F^P(a'-a,b'-b,(a,b).p)
e^{-2i\pi(\nu ab(b'-b)+\nu\frac{ a(b'-b)^2}{2}+\frac{\nu a^3(b'-b)}{6})}\xi(a,b)dadb.\\
\end{eqnarray}
\begin{definition}
 Let as before $P=\exp(\p)$ and $\pi_0=\ind P{G_{5,6}}\chi_0$ be the left regular representation of $G_{5,6}$ on the Hilbert space 
$L^2(G_{5,6}/P)\simeq L^2(\R^2).$ Then the image $\pi_0(C^*(G_{5,6}))$ is just the $C^*-$algebra of $\R^2$ considered as an algebra 
of convolution operators on $L^2(\R^2)$ and $\pi_0(C^*(G_{5,6}))$ is isomorphic to the algebra $C_0(\R^2)$ of continuous functions vanishing at 
infinity on $\R^2$ via the abelian Fourier transform 
$$\widehat F(a,b):=\int_{G_{5,6}}F(g)e^{-2i\pi\ell_{a,b}(\log g)}dg,\ a,b\in\R, F\in L^1(G_{5,6}).$$
\end{definition}
\begin{definition}\rm
 The Fourier transform $\hat a=\F(a)$ of an element
$a\in C^*(G_{5,6})$ is defined as to 
 the field of bounded linear operators  over the dual space of $ G_{5,6} $, but
 where we put  all the unitary characters of $ G_{5,6}$ together to form 
 the representation $ \pi_0 $.   
This gives us the set $ \GA_0^{4}\cup\GA_1^4\cup\GA_2^{4}\cup\GA_3^{5,6} $ and we define for
 $ a \in C^*(G_{5,6}) $ the operator field:
\begin{eqnarray}\label{he7}
\nn&&\hat
 a(\nu)=\F(a)(\nu):=\pi_{\nu}(a)\in\K(\l2{\R^{2}}),\ \nu\in\GA_3^{5,6};
\\
\nn &&\hat a(\be,\mu)=\F(a)(\be,\mu):=\pi_{\be,\mu}(a)
\in\K(\l2\R),\ (\be,\mu)\in\GA_2^{4};\\
\nn &&\hat a(\rh)=\F(a)(\rh):=\pi_{\rh}(a)
\in\K(\l2\R),\ \rh\in\GA_1^4;\\
\nn&&\hat a(0)=\F(a)(0):=\pi_0(a)\in C^*(\R^2)\subset \B(\l2{\R^2}).
\end{eqnarray}
\end{definition}
\begin{theorem}\label{normcontG56}
 The mapping $\GA_3^{5,6}\mapsto\B(L^2(\R^2)):\ell\to\pi_\ell(F)$  is norm-continuous  for all $F\in C^*(G_{5,6}).$
\end{theorem}
\begin{theorem}
 For every sequence $(\O_{\ell_k})_k\subset\g^*_{5,6}/G_{5,6}$ going to infinity we have that $$\lim_k\noop{\pi_{\ell_k}(F)}=0.$$
\end{theorem}
\subsection{ Passing from $\GA_3^{5,6}$ to $\GA_2^{4}\cup\GA_1^4\cup\GA_0^4$.}
Let $\ol\O=(\O_{\ell_k})_k\subset\GA_3^{5,6}$ be a properly converging sequence where $\ell_k=(0,0,0,0,\nu_k),\ k\in\N$ such that 
$\lim_k\nu_k=0.$ Let $p_k:=(\ell_k)\res\p.$ By Theorem \ref{topodualG56} the restriction of the limit set $L(\ol\O)$ to $\p$ 
is the closed set $L=L(\O)\res\p=\{(\rh,\mu,0),\ \rh\in\R, \mu\in\R\}.$
\begin{definition}
 For $k\in\N$ let:
 \begin{eqnarray}\label{defsetG56}
\nn\ve_k&:=&\val{\nu_k}^{\frac{3}{4}},\\
\nn I_{i,j}^k&:=&\left\{(c,u,\nu_k)\in p^*;
i\ve_k^{\frac 1 4}-\frac{j^2\ve_k^2}{2\nu_k} \leq{c-\frac{u^2}{2\nu_k}}< i\ve_k^{\frac 1 4}-\frac{j^2\ve_k^2}{2\nu_k}+\ve_k^{\frac 1 2}
\text{ and }j\ve_k\leq u<j\ve_k+\ve_k\right\},\\
\nn U_{i,j}^k&:=&\left\{(x,y)\in \R^2;(xA+yB)\cdot p_k\in I_{i,j}^k\right\},\ j\in\Z.
\end{eqnarray}
Finally:
\begin{eqnarray*}
 U^k&:=&\underset{i,j\in\Z}{\bigcup} U_{i,j}^k.
\end{eqnarray*}
Choose now the sequence $R_k,$ such that $\lim_kR_k=+\iy,\ \lim_kR_k\de_k=0.$
Let also for $k\in\N,\ i,j\in\Z:$
\begin{eqnarray}\label{defseqG56}
\nn& & x_{j}^k :=\frac{j\ve_k}{\nu_k},\ y_{i,j}^k:=\frac{(x_{j,4}^k)^2}{2}+\frac{i\ve_k^{\frac 1 4}}{\nu_k},
\ g_{i,j}^k=x_{j}^kA+y_{i,j}^kB.
\end{eqnarray}
Let for $i,j\in\Z,\ k\in\N^*:$
\begin{eqnarray*}
& &p_{i,j}^k:=(i\ve_k^{\frac{1}{4}},j\ve_k,0).
\end{eqnarray*}
\end{definition}
An easy computation gives:
\begin{eqnarray}\label{calcg.pG56}
 \nn& & g_{i,j}^k\cdot p_k=(i\ve_k^{\frac{1}{2}},j\ve_k,\nu_k)=p_{i,j}^k+(0,0,\nu_k).
\end{eqnarray}
\begin{proposition}\label{suppcompG56}$ $
  Let $K$ be a compact subset, for $k$ large enough we have that 
  $$K U_{i,j}^k\subset \underset{i',j'=-1}{\overset{1}{\bigcup}}U_{i'+i,j'+j}^k=:V_{i,j}^k.$$
\end{proposition}
\begin{proof}$ $
  We can suppose that $K P$ is contained in $[-M,M]^2P$ for some $M>0.$
For $r=(u,v)\in K P\subset  G_{5,6}/P$ and $s=(x,y)\in U^k$ we have that 
\begin{eqnarray*}
 (rs).p_k&=&(\nu_kv+\nu_ky+\nu_kxu+\nu_k\frac{x^2}{2}+\nu_k\frac{u^2}{2},\nu_kx+\nu_k u,\nu_k)
\end{eqnarray*}
we have 
\begin{eqnarray*}
 \nn& & (x,y)\in U_{i,j}^k\\
\nn &\Leftrightarrow&(xA+yB)\cdot p_k\in I^{k}_{i,j}\\
\nn &\Rightarrow&
\left\{
\begin{array}{c}
j\ve_k\leq \nu_kx<j\ve_k+\ve_k,\\
i\ve_k^{\frac 1 4}-\frac{j^2\ve_k^2}{2\nu_k} \leq\nu_ky< i\ve_k^{\frac 1 4}-\frac{j^2\ve_k^2}{2\nu_k}+\ve_k^{\frac 1 2},
\end{array}
\right.\\
\nn &\Rightarrow&
\left\{
\begin{array}{c}
(j-1)\ve_k\leq \nu_kx+\nu_ku<(j+1)\ve_k+\ve_k,\\
(i-1)\ve_k^{\frac 1 4}-\frac{j^2\ve_k^2}{2\nu_k} \leq\nu_ky+\nu_kv< (i+1)\ve_k^{\frac 1 4}-\frac{j^2\ve_k^2}{2\nu_k}+\ve_k^{\frac 1 2}.
\end{array}
\right.
 \end{eqnarray*}
It follows that $K U_{i,j}^k\subset \underset{i',j'=-1}{\overset{1}{\bigcup}}U_{i'+i,j'+j}^k.$
\end{proof}
\begin{definition}
 For $k\in\N^*$
   Let $$R^k=\left[-\frac{\ve_k}{\val{\nu_k}},\frac{\ve_k}{\val{\nu_k}}\right]\times
  \left[-\frac{\ve_k^{\frac{1}{2}}}{\val{\nu_k}},\frac{\ve_k^{\frac{1}{2}}}{\val{\nu_k}}\right].$$
\end{definition}
\begin{lemma}\label{lemma1G56}
 For $k\in\N^*$ large enough,  for any $i,j\in\Z$,  the set $U_{i,j}^k$ is contained in $ R^k+g_{i,j}^k$. 
\end{lemma}
\begin{proof}
 Let $s=(x,y)\in U_{i,j}^k$ Then:
 \begin{eqnarray*}
  & &(xA+yB)\cdot p_k\in I_{i,j}^k\\
  &\Longleftrightarrow& \begin{cases}
j\ve_k\leq \nu_kx<j\ve_k+\ve_k\Rightarrow\val{x-x_j^k}\leq\frac{\ve_k}{\val{\nu_k}}\Rightarrow 
x\in\left[-\frac{\ve_k}{\val{\nu_k}},\frac{\ve_k}{\val{\nu_k}}\right]+x_j^k,\\
i\ve_k^{\frac 1 4}-\frac{j^2\ve_k^2}{2\nu_k} \leq\nu_ky< i\ve_k^{\frac 1 4}-\frac{j^2\ve_k^2}{2\nu_k}+\ve_k^{\frac 1 2}
\Rightarrow\val{y-y_{i,j}^k}\leq \frac{\ve_k^{\frac 1 2}}{\val{\nu_k}}\Rightarrow 
y\in\left[-\frac{\ve_k^{\frac{1}{2}}}{\val{\nu_k}},\frac{\ve_k^{\frac{1}{2}}}{\val{\nu_k}}\right]+y_{i,j}^k.
                         \end{cases}\\
\nn&\Longrightarrow& s\in R^k+g_{i,j}^k.
\end{eqnarray*}

\end{proof}
\begin{lemma}\label{lemma2G56}
 For $k\in\N^*$ large enough, for $i,j\in\Z$  and any 
 $(x,y)\in U_{i,j}^k$ we have that 
 $$\no{(xA+yB)\cdot p_k-((xA+yB)\cdot (g_{i,j}^k)\inv)\cdot p_{i,j}^k}\leq4\ve_k^{\frac 1 2}.$$
\end{lemma}
\begin{proof}
 For $(x,y)\in U_{i,j}^k$ we have that $ (x,y)=(x'+x^k_{j},y'+y^k_{i,j}) $ where $ \val{\nu_kx'}\leq \ve_k $ 
 and $ \val{\nu_k y'}\leq  \ve_k^{{\frac{1}{2}}}$. Therefore
\begin{eqnarray}\label{}
\nn &&
\no{((x'+x^k_{j})A+(y'+y^k_{i,j})B)p_k-(x'A+y'B)\cdot p^{k}_{i,j}}\\
\nn &=&
\no{(\nu_ky'+\nu_ky_{i,j}^k+\nu_kx'x_j^k+\nu_k\frac{(x_j^k)^2}{2}+\nu_k\frac{x'^2}{2},\nu_kx_j^k+\nu_kx',\nu_k)
-(i\ve_k^{\frac 1 4}+j\ve_kx',j\ve_k,0)}\\ 
\nn &=&
\no{(\nu_ky'+i\ve_k^{\frac 1 4}-\frac{j^2\ve_k^2}{2\nu_k}+j\ve_kx'+\frac{j^2\ve_k^2}{2\nu_k}+\nu_k\frac{x'^2}{2},j\ve_k+\nu_kx',\nu_k)
-(i\ve_k^{\frac 1 4}+j\ve_kx',j\ve_k,0)}\\ 
\nn &=&
\no{(\nu_ky'+\nu_k\frac{x'^2}{2},\nu_kx',\nu_k)}\\
\nn&=& \val{\nu_ky'+\nu_k\frac{x'^2}{2}}+\val{\nu_kx'}+\val{\nu_k}\\
\nn&\leq&\ve_k^{\frac 1 2}+\frac{\val{\nu_k}^{\frac1 2}}{2}+\ve_k+\val{\nu_k}\\
\nn&\leq&4\ve_k^{\frac 1 2}.
\end{eqnarray}
\end{proof}
 \begin{definition}\label{vkdefG56}$  $\rm 
    Define for $ k\in \N$ and  $\ph\in C_{\ol\O}$ the linear
operator  $ \tilde \si_{k,\ol\O}(\ph)$ by
\begin{eqnarray}\label{deftildesikG56}
   \tilde\si_{k,\ol\O}(\ph):=\sum_{i\in \Z}\sum_{j\in \Z}M_{V^{k}_{i,j}}\circ
\tilde\si_{{(g^{k}_{i,j}})\inv\cdot p^{k}_{i,j}}(\ph)\circ M_{U^{k}_{i,j}}, 
\end{eqnarray}
where $ \tilde\si_\ell$ for $ \ell=(g^{k}_{i,j})\inv\cdot p^{k}_{i,j}, $ is an in Equation
(\ref{sielldefG53}). For $ a\in C^*(G_{5,6}) $ we have that 
$
 \si_{k,\ol\O}(a)=\tilde\si_{k,\ol\O}(\wh a\res {L(\ol O)}).
$ 
 \end{definition}
The proof of the next proposition is similar to that of Proposition \ref{progencondG53}.
\begin{proposition}\label{progencondG56}
 Let $a\in C^*(G_{5,6})$. Then: 
 \begin{eqnarray*}
  \underset{k\to\iy}{\lim}\noop{\pi_{\ell_k}(a)-\si_{k,\ol\O}(a)}=0.
 \end{eqnarray*}
\end{proposition}

We have treated now all simply connected, connected undecomposable Lie groups 
of dimension $ \leq 5 $. The other simply connected connected groups of 
dimension $ \leq 5 $ are of the form $ G_1\ti \R^d $, with $ G_1 $ 
undecomposable and $ \dim{G_1}+d\leq 5 $. It is easy to extend our methods to 
these groups to. We have thus  established the following theorem:
\begin{theorem}\label{finres}
The $C^* $-algebra of every connected nilpotent Lie group of dimension $\leq 5 $ has  norm controlled dual limits.
 \end{theorem}

\end{document}